\documentclass{amsart}
\usepackage{graphicx}
\usepackage{enumerate}
\usepackage{amsmath}
\usepackage{amssymb}
\usepackage[usenames,dvipsnames]{color}
\usepackage{pinlabel}

\newtheorem{thm}{Theorem}[section]
\newtheorem{cor}[thm]{Corollary}

\newtheorem{defin}[thm]{Definition}
\newtheorem{lemma}[thm]{Lemma}
\newtheorem{example}[thm]{Example}
\newtheorem{prop}[thm]{Proposition}
\newtheorem{conj}{Conjecture}
\newcommand{\aaa}{\mbox{$\alpha$}}

\newcommand{\sss}{\mbox{$\sigma$}}

\newcommand{\bdd}{\mbox{$\partial$}}

\def\zed{{\mathbb Z}}

\def\R{{\mathbb R}}
\def\im{\mathop{\rm Im}\nolimits}

\begin{document}

\title[Fibered knots, Property~2R and slice vs ribbon]{Fibered knots and potential counterexamples to the Property~2R and Slice-Ribbon Conjectures}

\author[R E Gompf]{Robert E. Gompf}
\address{\hskip-\parindent
Robert Gompf\\
Mathematics Department \\
University of Texas at Austin\\
1 University Station C1200\\
Austin TX 78712-0257, USA}
\email{gompf@math.utexas.edu}

\author[M Scharlemann]{Martin Scharlemann}
\address{\hskip-\parindent
        Martin Scharlemann\\
        Mathematics Department\\
        University of California\\
        Santa Barbara, CA 93106, USA}
\email{mgscharl@math.ucsb.edu}

\author[A Thompson]{Abigail Thompson}
\address{\hskip-\parindent
Abigail Thompson\\
Mathematics Department \\
University of California, Davis\\
Davis, CA 95616, USA}
\email{thompson@math.ucdavis.edu}

\thanks{Research partially supported by National Science Foundation grants.  Thanks also to Mike Freedman and Microsoft's Station Q for rapidly organizing a mini-conference on this general topic.}

\date{\today}

\begin{abstract} If there are any $2$-component counterexamples to the Generalized Property R Conjecture, a least genus component of all such counterexamples cannot be a fibered knot.  Furthermore, the monodromy of a fibered component of any such counterexample has unexpected restrictions.

The simplest plausible counterexample to the Generalized Property R Conjecture could be a $2$-component link containing the square knot.  We characterize all two-component links that contain the square knot and which surger to $\#_{2} (S^{1} \times S^{2})$.  We exhibit a family of such links that are probably counterexamples to Generalized Property~R. These links can be used to generate slice knots that are not known to be ribbon.

\end{abstract}

\maketitle

\section{Introduction}

Recall the famous Property R theorem, proven in a somewhat stronger form by David Gabai \cite{Ga1}:

\begin{thm}[Property R] \label{thm:PropR} If $0$-framed surgery on a knot $K \subset S^3$ yields $S^1 \times S^2$ then $K$ is the unknot.
\end{thm}

\noindent Problem~1.82 in Kirby's problem list \cite{Ki2} conjectures a generalization to links: If surgery on an $n$-component link $L$ yields the connected sum $\#_nS^1\times S^2$ then $L$ becomes the unlink after suitable handle slides. (See the next section for further discussion.) Although this Generalized Property~R Conjecture can be traced back more than two decades, it seems that no progress has appeared in the literature until now. The present paper studies the conjecture, focusing on the case of 2-component links. We conclude that the conjecture is probably false, and analyze potential counterexamples. These examples have relevance to the Slice-Ribbon Conjecture. We also propose a weaker version of the Generalized Property~R Conjecture that seems much more likely, and is related to the 4-dimensional smooth Poincar\'e Conjecture.

We concentrate on links $L=U\cup V$ where $U$ is fibered. Critical to the argument are developments in sutured manifold theory, developments that go beyond those used by Gabai in his original proof of Property R. In particular we use sutured manifold theory to show (Theorem~\ref{thm:main}) that any counterexample of this form to the Generalized Property~R Conjecture generates another 2-component counterexample for which one component has smaller genus than $U$. With this motivation, we assume $U$ has genus~2. We then use Heegaard theory to completely characterize such links $L$ satisfying the hypothesis of the Generalized Property~R Conjecture (Corollary~\ref{cor:genustwo}): After handle slides, $V$ is a special type of curve lying on the fiber of $U$. For the simplest plausible $U$, namely the square knot, this characterization becomes simple enough (Corollary~\ref{cor:enumerate}) to allow enumeration of all examples up to handle slides.

 \begin{figure}[ht!]
 \labellist
\small\hair 2pt
\pinlabel ${n \; \; strands}$ at 95 200
\pinlabel ${n \; \; strands}$ at 338 203
\pinlabel ${-1}$ at 25 112
\pinlabel ${+1}$ at 410 115
  \endlabellist
    \centering
    \includegraphics[scale=0.7]{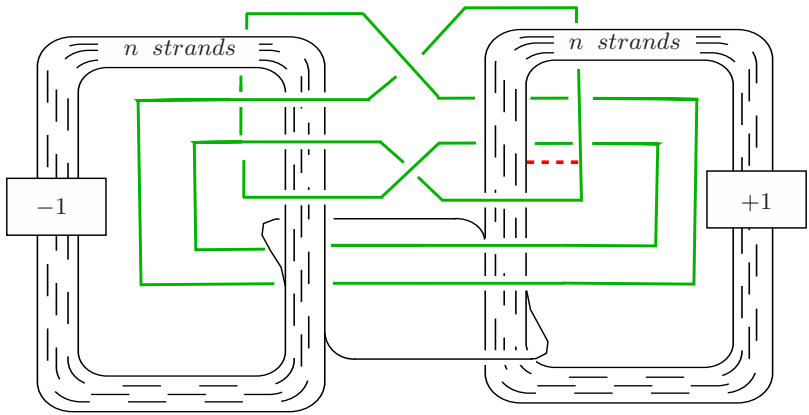}
    \caption{A counterexample to Generalized Property R?} \label{fig:Ln1}
    \end{figure}

We then use classical 4-manifold theory (the 20-year old paper \cite{Go1} of the first author) to argue that the square knot probably does lie in counterexamples to the Generalized Property~R Conjecture. The link $L_{n,1}$ in Figure~\ref{fig:Ln1} consists of a square knot interleaved with the connected sum of an $(n,n+1)$ torus knot with its mirror. (Each summand is depicted as an $n$-stranded spiral, with a full $\pm
1$-twist added relative to the plane of the paper. The dotted arc connecting the components will be used momentarily.)  We show in Section~\ref{sect:nonstandard} that $L_{n,1}$ is a counterexample to the Generalized Property~R Conjecture provided that the presentation $$\langle x,y \; | \; yxy=xyx,\ x^{n+1}=y^n\rangle$$ of the trivial group is Andrews-Curtis nontrivial, as is deemed likely by group theorists when $n\ge 3$. We will see (Section~\ref{slice}) that this link and its generalizations $L_{n,k}$ (Figure~\ref{fig:Lnk}) are necessarily slice, so we can generate slice knots just by band-summing the two components together. For $n\ge 3$ there is no apparent reason for the resulting knots to be ribbon. Thus, we obtain potential counterexamples to the Slice-Ribbon Conjecture. This appears to be the only known method for constructing such examples. Figure~\ref{fig:sliceknot} shows a simple potential counterexample obtained from $L_{3,1}$ by the band-move following the dotted arc in Figure~\ref{fig:Ln1}.

 \begin{figure}[ht!]
 \labellist
\small\hair 2pt
\pinlabel ${-1}$ at 22 112
\pinlabel ${+1}$ at 407 115
  \endlabellist
    \centering
    \includegraphics[scale=0.7]{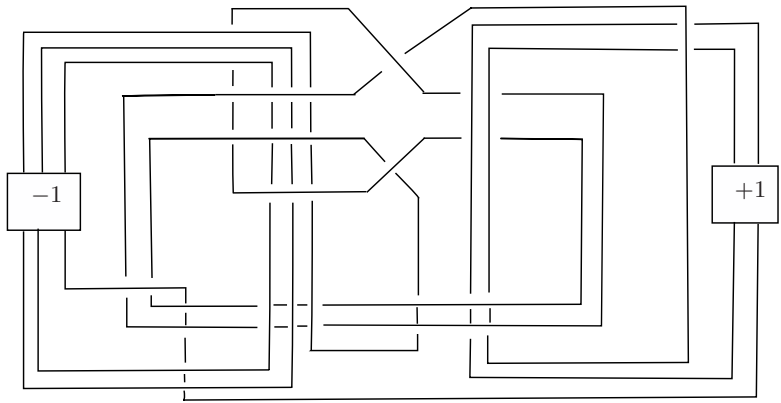}
    \caption{A slice knot that might not be ribbon} \label{fig:sliceknot}
    \end{figure}

The 4-dimensional viewpoint that discredits the Generalized Property~R Conjecture also shows how to rescue it by allowing more moves. In Section~\ref{sect:weakening} we discuss this more plausible version of the conjecture, and its close connection with the  4-dimensional smooth Poincar\'e Conjecture. Since the key new move is somewhat mysterious, Section~\ref{sect:Hopf} shows explicitly how it works to trivialize our potential counterexamples $L_{n,k}$. The main result, Proposition~\ref{prop:Ln}, can in principle be used to construct other potential counterexamples.

Unless explicitly stated otherwise, all manifolds throughout the paper will be smooth, compact and orientable.

\section{Handle slides and Generalized Property R}

To make sense of how Property R might be generalized, recall a small bit of $4$-manifold handlebody theory \cite{GS}.  Suppose $L$ is a link in a $3$-manifold $M$ and each component of $L$ is assigned a framing, that is a preferred choice of cross section to the normal circle bundle of the component in $M$. For example, if $M = S^3$, a framing on a knot is determined by a single integer, the algebraic intersection of the preferred cross-section with the longitude of the knot.  (In an arbitrary $3$-manifold $M$ a knot may not have a naturally defined longitude.)  Surgery on the link via the framing is standard Dehn surgery, though restricted to integral coefficients:  a regular neighborhood of each component is removed and then reattached so that the meridian is identified with the cross-section given by the framing.  Associated to this process is a certain $4$-manifold:  attach $4$-dimensional $2$-handles to $M \times I$ along $L \times \{ 1 \}$, using the given framing of the link components.  The result is a $4$-dimensional cobordism, called the {\em trace} of the surgery, between $M$ and the $3$-manifold $M'$ obtained by surgery on $L$. The collection of belt spheres of the $2$-handles constitute a link $L' \subset M'$ called the dual link; the trace of the surgery on $L \subset M$ can also be viewed as the trace of a surgery on $L' \subset M'$.

The $4$-manifold trace of the surgery on $L$ is unchanged if one $2$-handle is slid over another $2$-handle.  Such a handle slide is one of several moves allowed in the Kirby calculus \cite{Ki1}.  When the $2$-handle corresponding to the framed component $U$ of $L$ is slid over the framed component $V$ of $L$ the effect on the link is to replace $U$ by the band sum $\overline{U}$ of $U$ with a certain copy of $V$, namely the copy given by the preferred cross-section realizing the framing of $V$. If $M = S^3$ and the framings of $U$ and $V$ are given by the integers $m$ and $n$, respectively, then the integer for $\overline{U}$ will be $m+n\pm 2\cdot link(U,V)$ where $link$ denotes the linking number and the sign is $+$ precisely when the orientations of $U$ and the copy of $V$ fit together to orient $\overline{U}$. (This requires us to orient $U$ and $V$, but the answer is easily seen to be independent of the choice of orientations.) The terms {\em handle addition} and {\em subtraction} are used to distinguish the two cases of sign ($+$ and $-$, respectively) in the formula. Any statement about obtaining $3$-manifolds by surgery on a link will have to take account of this move, which we continue to call a handle-slide, in deference to its role in $4$-dimensional handle theory.

Suppose $\overline{U} \subset M$ is obtained from components $U$ and $V$ by the handle-slide of $U$ over $V$ as described above.  Let $U' \subset M'$ and $V' \subset M'$ be the dual knots to $U$ and $V$.  It will be useful to note this counterintuitive but elementary lemma:

\begin{lemma}  \label{lemma:dual}  The link in $M'$ that is dual to $\overline{U} \cup V$ is $U' \cup \overline{V'}$, where $\overline{V'}$ is obtained by a handle-slide of $V'$ over $U'$.
\end{lemma}

\begin{proof}  It suffices to check this for the simple case in which the $3$-manifold is a genus $2$ handlebody, namely a regular neighborhood of $U$, $V$, and the arc between them along which the band-sum is done.  A sketch of this is shown in Figure \ref{fig:dual2}.  The dual knots $U' = \overline{U}'$, $V'$ and $\overline{V}'$ are displayed as boundaries of meridian disks for regular neighborhoods of $U$, $V$ and $\overline{V} = V$ respectively.

 \begin{figure}[ht!]
 \labellist
\small\hair 2pt
\pinlabel \color{red}{$\overline{U}$} at 390 543
\pinlabel \color{red}{$U$} at 100 543
\pinlabel \color{black}{slide} at 171 512
\pinlabel \color{blue}{$\overline{V'}$} at 465 500
\pinlabel ${U'}$ at 90 435
\pinlabel ${\overline{U}' = U'}$ at 380 435
\pinlabel \color{ForestGreen}{$V$} at 245 545
\pinlabel \color{blue}{$V'$} at 225 435
 \color{black}
  \endlabellist
    \centering
    \includegraphics[scale=0.7]{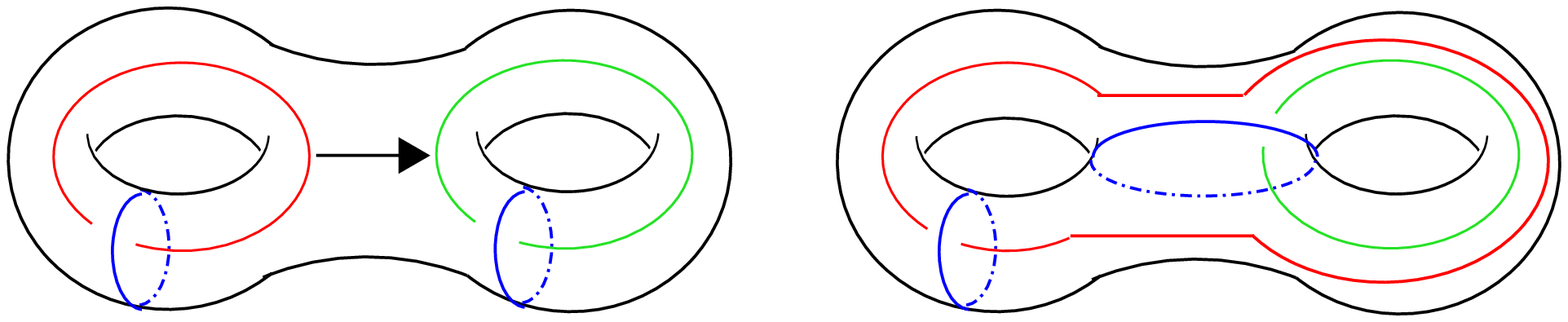}
    \caption{} \label{fig:dual2}
    \end{figure}

Alternatively, a $2$-dimensional schematic of the $4$-dimensional process is shown in Figure \ref{fig:dual}.  The handle corresponding to $U$ is shown half-way slid across the handle corresponding to $V$.  Each disk in the figure is given the same label as its boundary knot in $M$ or $M'$ as appropriate. \end{proof}

     \begin{figure}[ht!]
 \labellist
\small\hair 2pt
\pinlabel \color{red}{$\overline{U}$} at 37 183
\pinlabel \color{red}{$U$} at 137 86
\pinlabel \color{black}{slide} at 346 168
\pinlabel \color{blue}{$\overline{V'}$} at 238 122
\pinlabel ${U' = \overline{U}'}$ at 162 238
\pinlabel \color{ForestGreen}{$V$} at 310 50
\pinlabel \color{blue}{$V'$} at 396 101
 \color{black}
  \endlabellist
    \centering
    \includegraphics[scale=0.6]{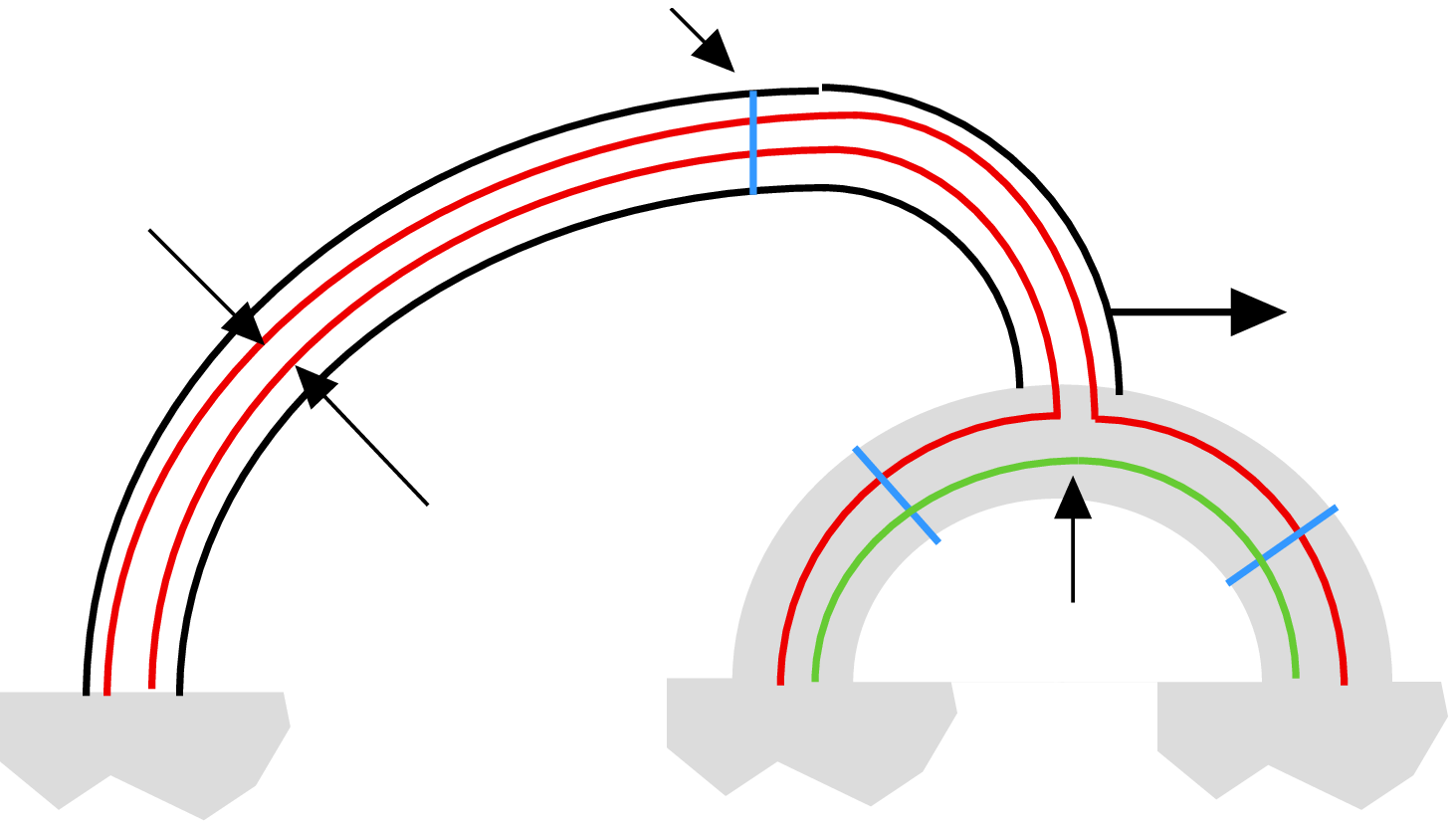}
    \caption{} \label{fig:dual}
    \end{figure}

Let $\#_{n} (S^{1} \times S^{2})$ denote the connected sum of $n$ copies of $S^1 \times S^2$.  The Generalized Property R conjecture (see \cite[Problem 1.82]{Ki2}) says this:

\begin{conj}[Generalized Property R] \label{conj:genR} Suppose $L$ is an integrally framed link of $n \geq 1$ components in $S^3$, and surgery on $L$ via the specified framing yields $\#_{n} (S^{1} \times S^{2})$.  Then there is a sequence of handle slides on $L$ that converts $L$ into a $0$-framed unlink.
\end{conj}

In the case $n = 1$ no slides are possible, so Conjecture \ref{conj:genR} does indeed directly generalize Theorem \ref{thm:PropR}.  On the other hand, for $n > 1$ it is certainly necessary to include the possibility of handle slides.  Figure \ref{fig:squareknot} shows an example of a more complicated link  on which $0$-framed surgery creates $\#_{2} (S^{1} \times S^{2})$.  To see this, note that the Kirby move shown, band-summing the square knot component to a copy of the unknotted component,  changes the original link to the unlink of two components, for which we know surgery yields $\#_{2} (S^{1} \times S^{2})$.  Even more complicated links with this property can be obtained, simply by using Kirby moves that complicate the link rather than simplify it.  See Figure \ref{fig:squareknot2b}; the free ends of the band shown can be connected in an arbitrarily linked or knotted way.

 \begin{figure}[ht!]
 \labellist
\small\hair 2pt
\pinlabel $0$ at 72 89
\pinlabel $0$ at 137 120
\pinlabel $0$ at 281 18
\pinlabel $0$ at 338 120
\pinlabel $0$ at 432 18
\pinlabel $0$ at 432 120
\pinlabel {band sum here} at 108 -11

 \endlabellist
    \centering
    \includegraphics[scale=0.7]{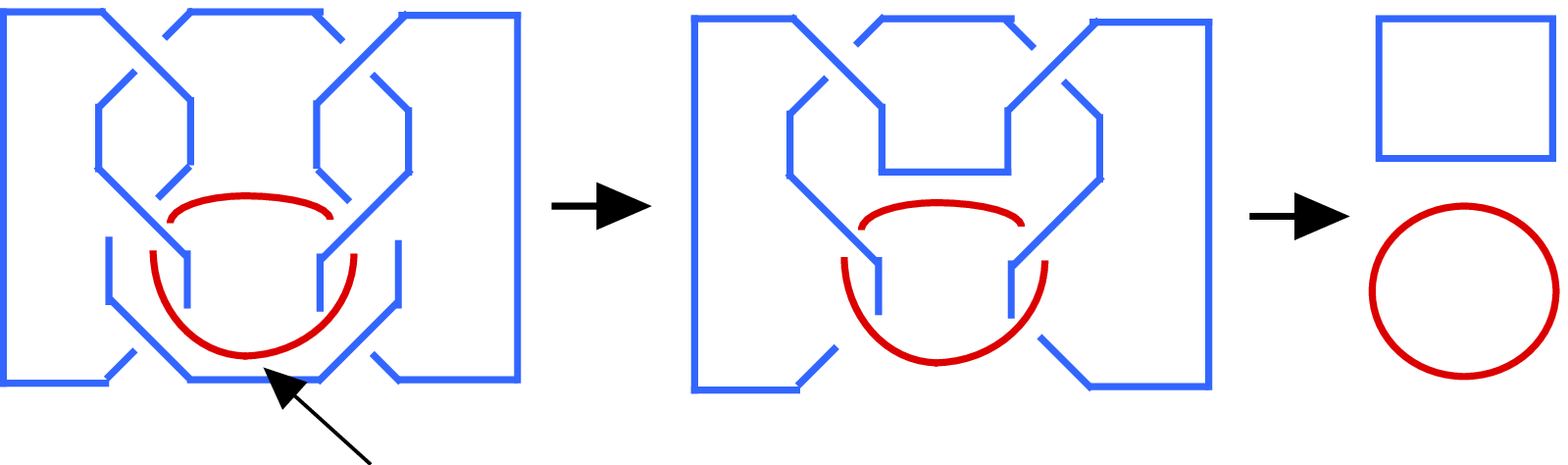}
    \caption{} \label{fig:squareknot}
    \end{figure}

 \begin{figure}[ht!]
    \centering
    \includegraphics[scale=0.7]{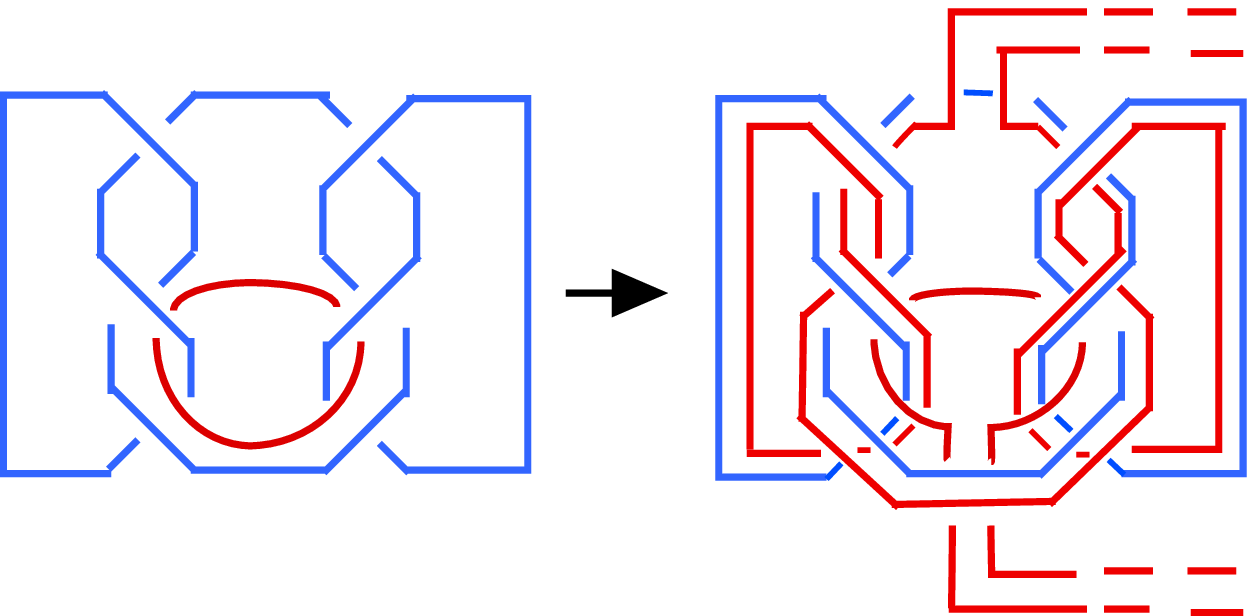}
    \caption{} \label{fig:squareknot2b}
    \end{figure}

The conjecture can be clarified somewhat by observing that the only framed links that are relevant are those in which all framings and linking numbers are trivial.  There is a straightforward $4$-dimensional proof, using the intersection pairing on the trace of the surgery.  Here is an equally elementary $3$-dimensional proof:

\begin{prop} \label{prop:frame}  Suppose $L$ is a framed link of $n \geq 1$ components in $S^3$, and surgery on $L$ via the specified framing yields $\#_{n} (S^{1} \times S^{2})$.  Then the components of $L$ are algebraically unlinked and the framing on each component is the $0$-framing.
\end{prop}

\begin{proof}  It follows immediately from Alexander duality that $H_1(S^3 - \eta(L)) \cong H^1(\eta(L)) \cong\mathbb{Z}^n.$  In particular, filling in the solid tori via whatever framing we are given yields an epimorphism, hence an isomorphism $H_1(S^3 - \eta(L)) \to H_1(\#_{n} (S^{1} \times S^{2}))$.  For each torus component $T$ of $\bdd \eta(L)$, the filling will kill some generator of $H_1(T)$, so the homomorphism $H_1(T) \to H_1(\#_{n} (S^{1} \times S^{2}))$ is not injective.  It follows that the homomorphism $H_1(T) \to H_1(S^3 - \eta(L))$ cannot be injective and, moreover, $ker(H_1(T) \to H_1(S^3 - \eta(L)))$ must contain the framing curve.  But $ker(H_1(T) \to H_1(S^3 - \eta(L)))$ must be contained in the subgroup generated by the standard longitude, since this is the only subgroup that is trivial when the other components of $\eta(L)$ are just left in place and not removed from $S^3$.  It follows that the framing at each component is that of the standard longitude, namely the $0$-framing.   Since the longitude of each $T$ is null-homologous in $H_1(S^3 - \eta(L))$ it follows that all linking numbers are trivial.
\end{proof}

There is also an immediate topological restriction on the link itself, which carries over to a restriction on the knots that can appear as individual components of such a link (see \cite[Theorem 2]{Hi}):

\begin{prop}[Hillman] \label{prop:slice}  Suppose $L$ is a framed link of $n \geq 1$ components in $S^3$, and surgery on $L$ via the specified framing yields $\#_{n} (S^{1} \times S^{2})$.  Then $L$ bounds a collection of $n$ smooth $2$-disks in a $4$-dimensional homotopy ball bounded by $S^3$.
\end{prop}

It follows from Freedman's proof of the $4$-dimensional topological Poincar\'e Conjecture \cite{Fr} that $L$ (and so each component of $L$) is topologically slice in $B^4$.  See also the earlier \cite{KM} for the case $n = 1$.

\begin{proof}  Consider the $4$-manifold trace $W$ of the surgery on $L$.   $\bdd W$ has one end diffeomorphic to $S^3$ and the other end, call it $\bdd_1 W$, diffeomorphic to $\#_{n} (S^{1} \times S^{2})$.  $W$ has the homotopy type of a once-punctured $\natural_{n} (B^{2} \times S^{2})$.  Attach $\natural_{n} (S^{1} \times B^{3})$ to $\bdd_1 W$ via the natural identification $\bdd B^{3} \cong S^2$.  The result is a homotopy $4$-ball, and the cores of the original $n$ $2$-handles that are attached to $L$ are the required $n$ $2$-disks.
\end{proof}

Somewhat playfully, we can turn the Generalized Property R Conjecture, which is a conjecture about links, into a conjecture about knots, and also stratify it by the number of components, via the following definition and conjecture.

\begin{defin} A knot $K \subset S^3$ has {\bf Property nR} if it does not appear among the components of any $n$-component counterexamples to the Generalized Property R conjecture.  That is, whenever $K$ is a component of an n-component link $L \subset S^3$ and some integral surgery on $L$ produces  $\#_{n} (S^{1} \times S^{2})$, then there is a sequence of handle slides on $L$ that converts $L$ into a $0$-framed unlink.
\end{defin}

\begin{conj}[Property nR Conjecture] All knots have Property nR.
\end{conj}

Thus the Generalized Property R conjecture is that the Property nR Conjecture is true for all $n \geq 1$.  Following Proposition \ref{prop:slice} any non-slice knot has Property nR for all $n$.  The first thing that we will show (Theorem \ref{thm:main}) is that if there are any counterexamples to Property~2R, a least genus such example cannot be fibered.  We already know that both of the genus one fibered knots (the figure 8 knot and the trefoil) cannot be counterexamples, since they are not slice, and we show below (Proposition \ref{prop:unknot}) that the unknot cannot be a counterexample. So these simplest of fibered knots do have Property~2R.  These considerations suggest that, if one is searching for a potential counterexample to Property nR and hope to exploit a knot's geometric properties, the simplest candidate could be the square knot, viewed as a potential counterexample to Property~2R.  After all, the square knot is symmetric, slice, fibered, and both its genus and its crossing number are low.

In fact it turns out that there is a relatively simple range of possibilities for what a $2$-component link containing the square knot and surgering to $\#_{2} (S^{1} \times S^{2})$ can look like.  The background for this comes from the theory of Heegaard splittings and is discussed in Sections \ref{Sect:Heeg} and \ref{sect:squareknot}.   In Section \ref{sect:nonstandard} we  present  strong evidence that Property~2R does indeed fail for the square knot.  Since the square knot is fibered, it would then follow from Theorem \ref{thm:main} that there is a counterexample to Property~2R among genus one knots.

On the other hand, the square knot may well satisfy a somewhat weaker property, which is still useful in the connection of Property~R to such outstanding $4$-manifold questions as the smooth Poincar\'e Conjecture and the smooth Schoenflies Conjecture.  For this weaker property, copies of canceling Hopf pairs may be added to the link before the handle slides and then removed after the handle slides. Beyond our specific example, we discuss the mechanics of how addition and later removal of a Hopf pair can be used to reduce the complexity of a surgery description.

\section{Special results for Property 2R}

Almost nothing is known about Generalized Property R, beyond the elementary facts noted in Propositions \ref{prop:frame} and \ref{prop:slice} that the framing and linking of the components of the link are all trivial and the links themselves are slice in a possibly exotic ball.
A bit more is known about Property~2R.  The first proposition was shown to us by Alan Reid:

\begin{prop}[A. Reid]  Suppose $L \subset S^3$ is a $2$-component link with tunnel number $1$.  If surgery on $L$ gives $\#_{2} (S^{1} \times S^{2})$ then $L$ is the unlink of two components.
\end{prop}

\begin{proof} The assumption that $L$ has tunnel number $1$ means that there is a properly embedded arc $\aaa \subset S^3 - \eta(L)$ so that  $S^3 - (\eta(L) \cup \eta(\aaa))$ is a genus $2$ handlebody $H$.  Let $G = \pi_1(S^3 - \eta(L))$.  There is an obvious epimorphism $Z*Z \cong \pi_1(H) \to G$ (fill in a meridian disk of $\eta(\aaa)$) and an obvious epimorphism $G \to \pi_1(\#_{2} (S^{1} \times S^{2})) \cong Z*Z$ (fill in solid tori via the given framing).  But any epimorphism $Z*Z \to Z*Z$ is an isomorphism, since free groups are Hopfian, so in fact $G \cong Z*Z$.  It is then a classical result that $L$ must be the unlink of two components. \end{proof}

This first step towards the Property~2R conjecture is a bit disappointing, however, since handle-slides (the new and necessary ingredient for Generalized Property R) do not arise.  In contrast, Figure \ref{fig:squareknot} shows that handle slides are needed in the proof of the following:

\begin{prop}  \label{prop:unknot}  The unknot has Property 2R.
\end{prop}

\begin{proof}  Suppose $L$ is the union of two components,  the unknot $U$ and another knot $V$, and suppose some surgery on $L$ gives $\#_{2} (S^{1} \times S^{2})$.  Following Proposition \ref{prop:frame} the surgery is via the $0$-framing on each and, since $U$ is the unknot, $0$-framed surgery on $U$ alone creates $S^1 \times S^2$.  Moreover, the curve $U' \subset S^1 \times S^2$ that is dual to $U$ is simply $S^1 \times \{ p \}$ for some point $p \in S^2$.

A first possibility is that $V$ is a satellite knot in $S^1 \times S^2$, so $V$ lies in a solid torus $K$ in such a way that the torus $\bdd K$ is essential in $S^1 \times S^2 - \eta(V)$.  Since there is no essential torus in $\#_{2} (S^{1} \times S^{2})$, $\bdd K$ compresses after the surgery on $V$.  Since $\#_{2} (S^{1} \times S^{2})$ contains no summand with finite non-trivial homology, it follows from the main theorem of \cite{Ga3} that $V$ is a braid in $K$ and that surgery on $V$ has the same effect on $S^1 \times S^2$ as some surgery on $K$.  Proposition \ref{prop:frame} shows that the surgery on $K$ must be along a longitude of $K$, but that would imply that $V$ has winding number $1$ in $K$.  The only braid in a solid torus with winding number $1$ is the core of the solid torus, so in fact $V$ is merely a core of $K$ and no satellite.  So we conclude that $V \subset S^1 \times S^2$ cannot be a satellite knot.

Consider the manifold $S^1 \times S^2 - \eta(V)$.  If it is irreducible, then it is a taut sutured manifold (see, for example, \cite{Ga1}) and two different fillings (trivial vs. $0$-framed) along $\bdd \eta(V)$ yield reducible, hence non-taut sutured manifolds.  This contradicts \cite{Ga2}.  We conclude that $S^1 \times S^2 - \eta(V)$ is reducible.  It follows that $V$ is isotopic \underline{in $S^1 \times S^2$} to a knot $\overline{V}$ lying in a $3$-ball in $S^1 \times S^2 - U'$ and that surgery on $\overline{V} \subset B^3$ creates a summand of the form $S^1 \times S^2$.  By Property R, we know that $\overline{V}$ is the unknot in $B^3$.  Hence $U \cup \overline{V} \subset S^3$ is the unlink of two components.

The proof, though, is not yet complete, because the isotopy of $V$ to $\overline{V}$ in $S^1 \times S^2$ may pass through $U'$.  But passing $V$ through $U' \subset S^1 \times S^2$ can be viewed as band-summing $V$ to the boundary of a meridian disk of $U'$ in $S^1 \times S^2$.  So the effect back in $S^3$ is to replace $V$ with the band sum of $V$ with a longitude of $U$.  In other words, the knot $\overline{V}$, when viewed in $S^3$, is obtained by from $V$ by a series of handle slides over $U$, a move that is permitted under Generalized Property R.
\end{proof}

In a similar spirit, the first goal of the present paper is to prove a modest generalization of Proposition \ref{prop:unknot}.   A pleasant feature is that, since the square knot is fibered, Figure \ref{fig:squareknot2b} shows that the proof will require handle slides of {\em both} components of the link.

\begin{thm} \label{thm:main} No smallest genus counterexample to Property 2R is fibered.
\end{thm}

\begin{proof}  Echoing the notation of Proposition \ref{prop:unknot}, suppose there is a $2$-component counterexample to Generalized Property R consisting of a fibered knot $U$ and another knot $V$.  Let $M$ be the $3$-manifold obtained by $0$-framed surgery on $U$ alone; $M$ will play the role that $S^1 \times S^2$ played in the proof of Proposition \ref{prop:unknot}.  Since $U$ is a fibered knot, $M$ fibers over the circle with fiber $F$, a closed orientable surface of the same genus as $U$.  The dual to $U$ in $M$ is a knot $U'$ that passes through each fiber exactly once.

The hypothesis is that $0$-framed surgery on $V \subset M$ creates $\#_{2} (S^{1} \times S^{2})$.  Following  \cite[Corollary 4.2]{ST}, either  the knot $V$ lies in a ball, or $V$ is cabled with the surgery slope that of the cabling annulus, or $V$ can be isotoped in $M$ to lie in a fiber, with surgery slope that of the fiber.  If $V$ were cabled, then the surgery on $K$ would create a Lens space summand, which is clearly impossible in $\#_{n} (S^{1} \times S^{2})$.  If $V$ can be isotoped into a ball or into a fiber, then, as argued in the proof of Proposition \ref{prop:unknot}, the isotopy in $M$ is realized in $S^3$ by handle-slides of $V$ over $U$, so we may as well regard $V$ as lying either in a ball that is disjoint from $U'$ or in a fiber $F \subset M$.  The former case, $V$ in a ball disjoint from $U'$ would, as in Proposition \ref{prop:unknot}, imply that the link $U \cup V \subset S^3$ is the unlink.  So we can assume that $V \subset F \subset M$.

The surgery on $V$ that changes $M$ to $\#_{2} (S^{1} \times S^{2})$ has this local effect near $F$:  $M$ is cut open along $F$, creating two copies $F_1, F_2$ of $F$, a $2$-handle is attached to  the copy of $V$ in each $F_i, i = 1, 2$, compressing the copies of the fiber to surfaces $F'_1, F'_2$.  The surfaces $F'_i$ are then glued back together by the obvious identification to give a surface $F' \subset \#_{2} (S^{1} \times S^{2})$. (See the Surgery Principle Lemma \ref{lemma:surgprin} below for more detail.)   This surface has two important features:  each component of $F'$ (there are two components if and only if $V$ is separating in $F$) has lower genus than $F$; and $F'$ intersects $U'$ in a single point.

Let $V' \subset \#_{2} (S^{1} \times S^{2})$ be the dual knot to $V$ and let $F''$ be the component of $F'$ that intersects $U'$.  $V'$ intersects $F'$ in some collection of points (in fact, two points, but that is not important for the argument).  Each point in $V' \cap F''$ can be removed by a handle-slide of $V'$ over $U'$ along an arc in $F''$.  Let $V''$ be the final result of these handle-slides.  Then $F''$ is an orientable surface that has lower genus than $F$, is disjoint from $V''$ and intersects $U'$ in a single point.

Following Lemma \ref{lemma:dual} the handle-slides of $V'$ over $U'$ in $\#_{2} (S^{1} \times S^{2})$ correspond in $S^3$ to handle-slides of $U$ over $V$.  Call the knot in $S^3$ that results from all these handle-slides $\overline{U} \subset S^3$.  Since $F''$ is disjoint from $V''$, and intersects $U'$ in a single meridian, $F'' - U'$ is a surface in $S^3 - \overline{U}$ whose boundary is a longitude of $\overline{U}$.  In other words, the knot $\overline{U}$, still a counterexample to Property~2R, has $$genus(\overline{U}) = genus(F'') < genus(F) = genus(U)$$ as required.
\end{proof}

\section{Fibered manifolds and Heegaard splittings} \label{Sect:Heeg}

We have just seen that a fibered counterexample to Property~2R would not be a least genus counterexample.  We now explore other properties of potential fibered counterexamples.  In this section we consider what can be said about the monodromy of a fibered knot in $S^3$, and the placement of a second component with respect to the fibering, so that surgery on the $2$-component link yields $\#_{2} (S^{1} \times S^{2})$.  Perhaps surprisingly, the theory of Heegaard splittings is useful in answering these questions.  Much of this section in fact considers the more general question of when $\#_{2} (S^{1} \times S^{2})$ can be created by surgery on a knot in a $3$-manifold $M$ that fibers over a circle.  The application to Property~2R comes from the special case in which the manifold $M$ is obtained from $0$-framed surgery on a fibered knot in $S^3$.

Suppose $F$ is a $2$-sided surface in a $3$-manifold $M$ and $c \subset F$ is an essential simple closed curve in $F$.  A tubular neighborhood $\eta(c) \subset M$ intersects $F$ in an annulus; the boundary of the annulus in $\bdd \eta(c)$ defines a slope on $\eta(c)$.  Let $M_{surg}$ denote the manifold obtained from $M$ by surgery on $c$ with this slope and let $F'$ be the surface obtained from $F$ by compressing $F$ along $c$.

\begin{lemma}[Surgery Principle] \label{lemma:surgprin}  $M_{surg}$ can be obtained from $M$ by the following $3$-step process:
\begin{enumerate}
\item Cut $M$ open along $F$, creating two new surfaces $F_1, F_2$ in the boundary, each homeomorphic to $F$.
\item Attach a $2$-handle to each of $F_1, F_2$ along the copy of $c$ it contains.  This changes each of the new boundary surfaces $F_1, F_2$ to a copy of $F'$.  Denote these two surfaces $F'_1, F'_2$.
\item Glue $F'_1$ to $F'_2$ via the natural identification.
\end{enumerate}
\end{lemma}

\begin{proof}  The surgery itself is a $2$-step process: Remove a neighborhood of $\eta(c)$, then glue back a copy of $S^1 \times D^2$ so that $\{point \} \times \bdd D^2$ is the given slope.  The first step is equivalent to cutting $F$ along an annulus neighborhood $A$ of $c$ in $F$, creating a torus boundary component as the boundary union of the two copies $A_1, A_2$ of $A$.  Thinking of $S^1$ as the boundary union of two intervals, the second step can itself be viewed as a two-step process:  attach a copy of $I \times D^2$ to each annulus $A_i, i = 1, 2$ along $I \times \bdd D^2$ (call the attached copies $(I \times D^2)_i$), then identify the boundary disks $(\bdd I \times D^2)_1$ with $(\bdd I \times D^2)_2$ in the natural way.  This creates a three-stage process which is exactly that defined in the lemma, except that in the lemma $F-A$ is first cut apart and then reglued by the identity. \end{proof}

The case in which $M$ fibers over a circle with fiber $F$ is particularly relevant.  We will stay in that case throughout the remainder of this section (as always, restricting to the case that $M$ and $F$ are orientable) and use the following notation:
\begin{enumerate}
\item $h: F \to F$ is the monodromy homeomorphism of $M$.
\item $c$ is an essential simple closed curve in $F$.
\item $F'$ is the surface obtained by compressing $F$ along $c$
\item $M_{surg}$ is the manifold obtained by doing surgery on $M$ along $c \subset F \subset M$ using the framing given by $F$.
\end{enumerate}
Note that $F'$ may be disconnected, even if $F$ is connected.

\begin{prop}  \label{prop:isotopic1}  Suppose $h(c)$ is isotopic to $c$ in $F$
\begin{itemize}
\item If $c$ is non-separating in $F$, or if $c$ is separating and the isotopy from $h(c)$ to $c$ reverses orientation of $c$, then $M_{surg} \cong N \# (S^1 \times S^2)$, where $N$ fibers over the circle with fiber $F'$.
\item If $c$ separates $F$ so $F' = F_a \cup F_b$, and the isotopy from $h(c)$ to $c$ preserves orientation of $c$, then $M_{surg} \cong M_a \# M_b$, where $M_a$ (resp. $M_b$) fibers over the circle with fiber $F_a$ (resp $F_b$).
\end{itemize}
\end{prop}

\begin{proof}  We may as well assume that $h(c) = c$ and consider first the case where $h|c$ is orientation preserving.  In this case, the mapping torus of $c$ in $M$ is a torus $T$ containing $c$.  The $3$-stage process of Lemma \ref{lemma:surgprin} then becomes:
\begin{enumerate}
\item $M$ is cut along $T$ to give a manifold $M_-$ with two torus boundary components $T_1, T_2$.  $M_-$ fibers over the circle with fiber a twice-punctured $F'$. ($F'$ is connected if and only if $c$ is non-separating.)
\item A $2$-handle is attached to each torus boundary component $T_i$, turning the boundary of $M_-$ into two $2$-spheres.
\item The two $2$-spheres are identified.
\end{enumerate}
The second and third stage together are equivalent to filling in a solid torus along each $T_i$, giving an $F'$-fibered manifold $M'$, then removing a $3$-ball from each solid torus and identifying the resulting $2$-spheres.  Depending on whether $F'$ is connected or not, this is equivalent to either adding $S^1 \times S^2$ to $M'$ or adding the two components of $M'$ together.

The case in which $h|c$ is orientation reversing is only slightly more complicated.  Since $M$ is orientable, the mapping torus of $h|c$ is a $1$-sided Klein bottle $K$, so $\bdd(M - \eta(K))$ is a single torus $T$.  The argument of Lemma \ref{lemma:surgprin} still mostly applies, since $c$ has an annulus neighborhood in $K$, and shows that the surgery can be viewed as attaching two $2$-handles to $T$ along parallel curves, converting the boundary into two $2$-spheres, then identifying the $2$-spheres.  This is again equivalent to filling in a solid torus at $T$ (which double-covers $K$) and then adding $S^1 \times S^2$.  But filling in a solid torus at $T \subset (M - \eta(K))$ changes the fiber from $F$ to $F'$.  (Note that if $c$ separates $F$, so $F' = F_a \cup F_b$, then since $h$ is orientation preserving on $F$ but orientation reversing on $c$, $h$ must exchange $F_a$ and $F_b$.  So $N$ also fibers over the circle with fiber $F_a \cong F_b$.)
\end{proof}
%
%It seems reasonable to hope that similar $3$-manifold arguments might show that all fibered knots have Property 2R, or at least Weak Property 2R.  More tools are at our disposal, for example:

\begin{cor}  \label{cor:isotopic2} If $M_{surg} \cong \#_{2} (S^{1} \times S^{2})$ and $h(c)$ is isotopic to $c$ in $F$, then $F$ is a torus.
\end{cor}

\begin{proof}   According to Proposition \ref{prop:isotopic1}, the hypotheses imply that $S^1 \times S^2$ fibers over the circle with fiber (a component of) $F'$.  But this forces $F' \cong S^2$ and so $F \cong T^2$. \end{proof}

Surgery on fibered manifolds also provides a natural connection between the surgery principle and Heegaard theory:

\begin{defin}  Suppose $H_1$, $H_2$ are two copies of a compression body $H$ and $h: \bdd_+ H  \to \bdd_+ H$ is a given homeomorphism.  Then the union of $H_1$, $H_2$ along their boundaries, via the identity on $\bdd_- H_i$ and via $h:  \bdd_+ H_1  \to \bdd_+ H_2$, is called the {\em Heegaard double} of $H$ (via $h$).
\end{defin}

Note that the closed complement $P$ of $\bdd_- H_1 = \bdd_- H_2$ in a Heegaard double is a manifold in which $H_1$ is attached to $H_2$ by only identifying $\bdd_+ H_1$ and $\bdd_+H_2$ (via $h$).  But such an identification constitutes a Heegaard splitting $ H_1 \cup_{\bdd_+} H_2$ of $P$.

Lemma \ref{lemma:surgprin} gives this important example:

\begin{example}  \label{example:double}  For $M, F, h, c, M_{surg}$ as above, let $H$ be the compression body obtained from $F \times I$ by attaching a $2$-handle to $F \times \{ 1 \} \subset F \times I$ along $c$.  Then $M_{surg}$ is the Heegaard double of $H$ via $h$.
\end{example}

Here is the argument:  Delete from $M$ regular neighborhoods of both the fiber $F$ that contains $c$ and another fiber $\hat{F}$.  The result is two copies $(F \times I)_i, i = 1, 2$ of $F \times I$ each containing a copy $F_i$ of $F$ and a copy $\hat{F}_i$ of $\hat{F}$.  One can reconstruct $M$ by identifying $F_1$ and $F_2$ by the natural identification and identifying $\hat{F}_1$ with $\hat{F}_2$ via $h$.  Now apply the surgery principle to the surface $F_1 = F_2$ in $M$.  Each $(F \times I)_i$ becomes a compression body $H_i$ in which $\hat{F}_i$ is $\bdd_+ H_i$ and the copy $F'_i$ of $F_i$ compressed along $c$ is $\bdd_- H_i$.

Next is a sample application, using Heegaard theory.  First a preliminary lemma:

\begin{lemma} \label{lemma:singlehandle}  Suppose $H$ is a compression body built with a single $2$-handle, whose core $E$ is bounded by a simple closed curve $c \subset \bdd_+H$.  Suppose $D \subset H$ is any properly embedded disk whose boundary is essential in $\bdd_+H$.
\begin{enumerate}
\item If $c$ is separating in $\bdd_+H$ then $\bdd D$ is isotopic to $c$ in $\bdd_+H$.
\item If $c$ is non-separating in $\bdd_+H$ then either $\bdd D$ is isotopic to $c$ in $\bdd_+H$ or $D$ cuts off a solid torus from $H$ whose meridian disk is bounded by $c$.
\end{enumerate}
\end{lemma}

\begin{proof} Dually, $H$ can be viewed as constructed by adding a single $1$-handle to, say, $\bdd_-H \times \{ 1 \} \subset \bdd_-H \times I$.  The co-core of the $1$-handle is the disk $E$, with boundary $c$.   A fairly standard innermost disk, outermost arc argument shows that $D$ may be isotoped to be disjoint from $E$: The central observation is that an outermost subdisk $D_0$ cut off from $D$ by $E$ will have boundary that is inessential in $\bdd_-H \times \{ 1 \}$, since $\bdd_-H \times \{ 1 \}$ is incompressible in $H - \eta(E) \cong  \bdd_-H \times I$, and the disk that $\bdd D_0$ bounds in $\bdd_-H \times \{ 1 \}$ can't contain an entire copy of $E$, since $\bdd D$ intersects both copies of $\bdd E \subset \bdd_-H \times \{ 1 \}$ in the same number of points.  (These copies of $E$ can be viewed as the two disks on which the original $1$-handle was attached.)

Once $D$ is disjoint from $E$, it lies entirely in $\bdd_-H \times I$ with $\bdd D \subset \bdd_-H \times\{ 1 \}$.  Since $\bdd_-H \times I$ has incompressible boundary, $\bdd D$ is inessential in $\bdd_-H \times\{ 1 \}$; since $\bdd D$ was essential in $\bdd_+H$ we conclude that $\bdd D$ bounds a disk in $\bdd_-H \times\{ 1 \}$ containing either one or both copies of $E$ that lie in $\bdd_-H \times\{ 1 \}$.    In the first case (which always occurs if $c$ is separating, so the copies of $E$ lie on different components of $\bdd_-H \times \{1 \}$), $D$ is parallel to $\bdd E$ in $\bdd_+(H_i)$.  In the second case, when $\bdd D$ bounds a disk in $\bdd_-H \times \{1 \}$ containing both copies of $E$, $D$ is separating and cuts off a solid torus from $H$ that contains $E$ as a meridian disk.
\end{proof}

\begin{prop}  \label{prop:monodromy} For $M, F, h, c, M_{surg}$ as above, suppose some surgery on $c$ gives a reducible manifold.  Then the surgery slope is that of $F$ and either
\begin{enumerate}
\item $h(c)$ can be isotoped in $F$ so that it is disjoint from $c$ or
\item $c \subset F$ is non-separating and $M_{surg} \cong N \# L$, where
\begin{itemize}
\item $N$ fibers over the circle with fiber $F'$ and
\item $L$ is either $S^3$ or a Lens space.
\end{itemize}
\end{enumerate}

\end{prop}

Note in particular that possibility (2) is not consistent with $M_{surg} \cong \#_{2} (S^{1} \times S^{2})$.

\begin{proof}   Choose distinct fibers $\hat{F}, F$ in $M$, with $c \subset F$.  Via \cite[Corollary 4.2]{ST} and the proof of Theorem \ref{thm:main} we know that the surgery on $c$ must use the framing given by the fiber $F$, so the result of surgery is $M_{surg}$.   Example \ref{example:double} shows that $M_{surg}$ is a Heegaard double via $h$, so the complement $M_- = M_{surg} - \eta(F')$ of a regular neighborhood of $F' = \bdd_- H$ has a Heegaard splitting $H_1 \cup_{\hat{F}} H_2$.  That is,  $\hat{F} = \bdd_+ H_1 = \bdd_+ H_2$.

If $F' \cong S^2$, so $F \cong T^2$, then the compression bodies $H_1$ and $H_2$ in the Heegaard splitting of $M_-$ are each punctured solid tori and $M_{surg}$ is obtained from $M_-$ by identifying the $2$-spheres at the punctures.  Hence $M_{surg}$ is the connected sum of $S^1 \times S^2$ with a closed $3$-manifold of Heegaard genus one, either $S^3$, a Lens space, or $S^1 \times S^2$.  But the last happens only if the same curve in $\hat{F}$ compresses in both $H_1$ and $H_2$; in our context, that implies $c$ and $h(c)$ are isotopic in $F$, and so can be isotoped to be disjoint.

If $F' \ncong S^2$, choose a reducing sphere for $M_{surg}$ with a minimal number of intersection curves with $F'$.  If the reducing sphere is disjoint from $F'$, then $M_-$ is reducible. If, on the other hand, the reducing sphere intersects $F'$, then at least one copy of $F'$ in $\bdd M_-$ must be compressible in $M_-$.  We conclude that in either case the Heegaard splitting $H_1 \cup_{\hat{F}} H_2$ of $M_-$ is weakly reducible (and possibly reducible), see \cite{CG}.  That is, there are essential disjoint simple closed curves $\aaa_1, \aaa_2$ in $F = \bdd_+ H_i$ which compress respectively in $H_1$ and $H_2$.

 \medskip

{\bf Case 1:}  The curve $c$ is separating.

In this case, since the compression bodies $H_i$ each have only the $2$-handle with attaching circle $c \subset F$ attached, it follows from Lemma \ref{lemma:singlehandle} that any curve in $\bdd _+ H_i = \hat{F}$ that compresses in $H_i$ is isotopic to $c \subset \bdd_+ H_i \cong F$.  In particular, fixing the identification $\hat{F} = \bdd_+ H_2$, $\aaa_2$  must represent $c$ in $\hat{F}$ and $\aaa_1$ represents $h(c)$.  Hence $c$ and $h(c)$ are disjoint.

\medskip

{\bf Case 2:}  The curve $c$ is non-separating, and so is at least one of the curves $\aaa_1, \aaa_2$.

If both curves $\aaa_i$ are non-separating then it follows from Lemma \ref{lemma:singlehandle} that both $\aaa_1$ and $\aaa_2$, when viewed in the handlebodies $H_1, H_2$, must each be isotopic to $c \subset \bdd_+ H_i \cong F$ and the case concludes as Case 1 did.

If $\aaa_2$ is non-separating, and $\aaa_1$ is separating, then it follows from Lemma \ref{lemma:singlehandle} that $\aaa_2$  is isotopic to $c \subset \bdd_+ H_2 = \hat{F}$ whereas $\aaa_1$ bounds a punctured torus $T \subset \bdd_+ H_2$ on which $h(c)$ lies.  If $\aaa_2$ is disjoint from $T$, then $c$ and $h(c)$ are disjoint, as required.  If $\aaa_2$ lies in $T$ then $\bdd T$ also bounds a disk in $H_2$.   That is, each $H_i$ can be viewed as the boundary sum, along a disk whose boundary is $\bdd T$ of a copy of $F' \times I$ and a solid torus.  When the $\hat{F}_i$ are identified to get $M_-$, the union of the disks in $H_1$ and $H_2$ bounded by $\bdd T$ constitute a sphere that decomposes $M_-$ into $F' \times I \# W$, where $W$ is Heegaard split by $T$ into two solid tori, with meridian disks bounded by $c$ and $h(c)$ respectively.  $M_{surg}$ is then obtained by identifying the two copies of $F'$ in $\bdd M_-$.  Hence $M_{surg} \cong N \# W$, where $N$ fibers over $S^1$ with fiber $F'$ .  If $|c \cap h(c)| > 1$ then $W$ is a Lens space. If $|c \cap h(c)| = 1$ then $W = S^3$.  If $|c \cap h(c)| = 0$ then $h(c)$ is disjoint from $c$.

\medskip

{\bf Case 3:} The curve $c$ is non-separating, but both $\aaa_1, \aaa_2$ are separating.

In this case, much as in Case 2, each $\aaa_i$ cuts off a torus $T_i$ from $\bdd_+ H_2 = \hat{F}$, with $c \subset T_2$ and $h(c) \subset T_1$.  Since the $\aaa_i$ are disjoint, the two tori either also are disjoint (and the proof is complete) or the two tori coincide.  If the two tori coincide, the argument concludes as in Case 2.  \end{proof}

\section{Could there be fibered counterexamples of genus $2$?}

In applying Proposition \ref{prop:monodromy} to the manifold $M$ obtained from $0$-framed surgery on a fibered knot $K \subset S^3$, note that the isotopy in the Proposition takes place in a fiber $F$ of $M$, the closed manifold obtained by $0$-framed surgery on $K$, not in the fiber $F - \{ point \}$ of the knot $K$ itself.   The importance of the distinction is illustrated by the following Proposition which, without the distinction, would (following Propositions \ref{prop:unknot}, \ref{prop:monodromy} and Corollary \ref{cor:genustwo}(1)) seem to guarantee that all genus $2$ fibered knots have Property~2R.

\begin{prop}  \label{prop:strictmono} Suppose $U \subset S^3$ is a fibered knot, with fiber the punctured surface $F_- \subset S^3$ and monodromy $h_-: F_- \to F_-$.  Suppose a knot $V \subset F_-$ has the property that $0$-framed surgery on the link $U \cup V$ gives $\#_{2} (S^{1} \times S^{2})$ and $h_-(V)$ can be isotoped to be disjoint from $V$ in $F_-$.  Then either $V$ is the unknot or $genus(F_-) \neq 1, 2$.  \end{prop}

\begin{proof}  {\em Case 1:} $V$ bounds a disk in $F_-$ or is parallel in $F_-$ to $\bdd F_- = U$.

In this case, $0$-framed surgery on $U \cup V$ would be $N \# S^1 \times S^2$, where $N$ is the result of $0$-framed surgery on $U$.  Our hypothesis is that $N \cong S^1 \times S^2$ which, by classical Property R \cite{Ga1}, implies that $U$ is the unknot.  Hence $genus(F_-) = 0$.

\bigskip

{\em Case 2:} $V$ is essential in $F_-$.

If $F_-$ is a punctured torus, then the fact that $V$ is essential and $h_-(V)$ can be isotoped off of $V$ imply that $h_-(V)$ is isotopic to $V$, and we may as well assume that $h_-(V) = V$.  The mapping torus of $h_-|V$ is then a non-separating torus in $S^3$, which is absurd.

Suppose $F_-$ is a punctured genus $2$ surface, and let $F$ denote the closed surface obtained by capping off the puncture.  We may as well assume that $h_-(V) \cap V = \emptyset$, and, following Corollary \ref{cor:isotopic2}, $h(V)$ is not isotopic to $V$ in $F$.  In particular, $V$ must be non-separating.   Since $V$ and $h(V)$ are non-separating and disjoint in $F_-$, but not isotopic in $F$, if $F_-$ is compressed along both $V$ and $h(V)$ simultaneously, $F_-$ becomes a disk.  Apply the Surgery Principle Lemma \ref{lemma:surgprin} to $V$ and conclude that $U \subset S^3$ bounds a disk after $0$-framed surgery on $V$.  In particular, if $N$ is the $3$-manifold obtained by $0$-framed surgery on $V$ alone, then surgery on $U \cup V$ would give $N \# S^1 \times S^2$.  For this to be $\#_{2} (S^{1} \times S^{2})$ would require $N \cong S^1 \times S^2$ hence, again by classical Property R, $V \subset S^3$ would be the unknot.
\end{proof}

Return to the general case of fibered manifolds and surgery on a curve $c$ in the fiber, and consider the case in which the fiber has genus two.  According to Corollary \ref{cor:isotopic2}, if the result of surgery on $c$ is $\#_{2} (S^{1} \times S^{2})$, then $h(c)$ is not isotopic to $c$ in $F$.   The following Proposition is a sort of weak converse.

\begin{prop}  \label{prop:nonisotopic} For $M, F, h, c, M_{surg}$ as above, suppose $F$ has genus $2$ and $h(c)$ can be isotoped off of $c$ in $F$.
If $h(c)$ is not isotopic to $c$ in $F$ then $M_{surg} \cong L \# S^1 \times S^2$, where $L$ is $S^3$, $S^1 \times S^2$, or a Lens space.
\end{prop}

\begin{proof}  We may as well assume that $h(c)$ is disjoint from $c$ but not isotopic to $c$ in $F$.  Since $F$ is genus two, this  immediately implies that $c$ is non-separating.

Take the Heegaard viewpoint of Example \ref{example:double}. The complement $M_-$ of a regular neighborhood of  $F'$ in $M_{surg}$ has a Heegaard splitting $H_1 \cup_{\hat{F}} H_2$, with the splitting surface $\hat{F}$ a fiber not containing $c$.  Since $h(c)$ can be isotoped off of $c$ in $\hat{F}$, the Heegaard splitting is a weakly reducible splitting, with $c \subset \hat{F}_2 = \bdd_+ H_2$ bounding a disk in $H_2$ and $h(c) \subset \hat{F}_2$ bounding a disk in $H_1$.

Now do a weak reduction of this splitting.  That is, consider the $2$-handles $C_2 \subset H_2$ with boundary $c \subset \hat{F}_2 = \bdd_+ H_2$ and $C_1 \subset H_1$ with boundary $h(c)$. Since $c$ and $h(c)$ are disjoint, $M_-$ can also be regarded as the union of compression bodies $H'_2 = (H_2 - \eta(C_2)) \cup \eta(C_1)$ and $H'_1 = (H_1 - \eta(C_1)) \cup \eta(C_2)$.  Each $H'_i$ can be regarded as obtained from $F' \times I$ by attaching a single $2$-handle; e. g. $H'_1$ is obtained from $F' \times I$ by attaching the $2$-handle $\eta(C_2)$ at one end.  Since $F$ was of genus two, $F'$ is a torus and so each $H'_i$ is a once-punctured solid torus.  Hence $M_{surg}$ is a Heegaard double of a once-punctured solid torus, and this is easily seen to be the connected sum of $S^1 \times S^2$ with a manifold $L$ that has a genus one Heegaard splitting.  \end{proof}

\begin{cor}  \label{cor:nonisotopic} For $M, F, h, c, M_{surg}$ as above, suppose $F$ has genus $2$ and $M_{surg}$  is reducible.

If $h(c)$ is not isotopic to $c$ in $F$ then $M_{surg} \cong L \# M'$, where $L$ is $S^3$, $S^1 \times S^2$, or a Lens space and $M'$ is either $S^1 \times S^2$ or a torus bundle over the circle.
\end{cor}
\begin{proof}  Via \cite[Corollary 4.2]{ST} and the proof of Theorem \ref{thm:main} we know that the surgery on $c$ must use the framing given by the fiber in which it lies.  Apply Proposition \ref{prop:monodromy}.  If the first conclusion holds, and $h(c)$ can be isotoped off of $c$ in $F$, then  Proposition \ref{prop:nonisotopic} can be applied and that suffices.  If the second conclusion holds then $c$ is non-separating, so $F'$ is a torus, as required. \end{proof}

\begin{cor}  \label{cor:genustwo} Suppose $U \subset S^3$ is a genus two fibered knot and $V \subset S^3$ is a disjoint knot.  Then $0$-framed surgery on $U \cup V$ gives $\#_{2} (S^{1} \times S^{2})$  if and only if  after possible handle-slides of $V$ over $U$,
\begin{enumerate}
\item $V$ lies in a fiber of $U$;
\item in the closed fiber $F$ of the manifold $M$ obtained by $0$-framed surgery on $U$, $h(V)$ can be isotoped to be disjoint from $V$;
\item $h(V)$ is not isotopic to $V$ in $F$; and
\item the framing of $V$ given by $F$ is the $0$-framing of $V$ in $S^3$.
\end{enumerate}
\end{cor}

\begin{proof}  Suppose first that $0$-framed surgery on $U \cup V$ gives $\#_{2} (S^{1} \times S^{2})$.  Apply  \cite[Corollary 4.2]{ST} as in the proof of Theorem \ref{thm:main} to handle-slide $V$ over $U$ until it lies in the fiber of $U$ in a way that the $0$-framing on $V$ is the framing given by the fiber in which it lies.  Proposition \ref{prop:monodromy} shows that $h(V)$ satisfies the second condition and Corollary \ref{cor:isotopic2} gives the third: $h(V)$ is not isotopic in $F$ to $V$.

For the other direction, suppose $V$ lies in a fiber of $U$ and the four conditions are satisfied.  The last condition says that the surgery on $V$ is via the slope of the fiber.  By Proposition \ref{prop:nonisotopic}, the surgery gives $L \# S^1 \times S^2$, for $L$ either $S^3$, a Lens space, or $S^1 \times S^2$.  But $U$ and $V$ are algebraically unlinked in $S^3$ (push $V$ off of $F$), so $0$-framed surgery on $U \cup V$ must give a homology $\#_{2} (S^{1} \times S^{2})$.  This forces $L$ to be a homology $S^1 \times S^2$, hence $S^1 \times S^2$ precisely. \end{proof}

\section{The square knot in links that surger to $\#_{2} (S^{1} \times S^{2})$} \label{sect:squareknot}

The square knot $Q$ is the connected sum of two trefoil knots, $K_R$ and $K_L$, respectively the right-hand trefoil knot and the left-hand trefoil knot.   There are many $2$-component links containing $Q$ so that surgery on the link gives $\#_{2} (S^{1} \times S^{2})$.  Figure \ref{fig:squareknot} shows (by sliding $Q$ over the unknot) that the other component could be the unknot; Figure \ref{fig:squareknot2b} shows (by instead sliding the unknot over $Q$) that the second component could be quite complicated.  In this section we show that, up to handle-slides of $V$ over $Q$, there is a straight-forward description of all two component links $Q \cup V$, so that surgery on $Q \cup V$ gives $\#_{2} (S^{1} \times S^{2})$.

The critical ingredient in the characterization of $V$ is the collection of properties listed in Corollary \ref{cor:genustwo}, which apply because $Q$ is a genus two fibered knot.   Let $M$ be the $3$-manifold obtained by $0$-framed surgery on the square knot $Q$, so $M$ fibers over the circle with fiber the closed genus $2$ surface $F$.  There is a simple picture of the monodromy $h: F \to F$ of the bundle $M$, obtained from a similar picture of the monodromy on the fiber of a trefoil knot, essentially by doubling it \cite[Section 10.I]{Ro}:

 \begin{figure}[ht!]
  \labellist
\small\hair 2pt
%\pinlabel $F$ at 210 700
%\pinlabel $P$ at 410 670
%\pinlabel \color{red}{$\sss$} at 266 607
%\pinlabel \color{ForestGreen}{$\rho$} at 195 770
%\pinlabel \color{Brown}$\gamma/\rho$ at 360 660
%\pinlabel \color{black}{$/\rho$} at 298 612

\pinlabel $F$ at 100 260
\pinlabel $P$ at 300 230
\pinlabel \color{red}{$\sss$} at 156 167
\pinlabel \color{ForestGreen}{$\rho$} at 90 330
\pinlabel \color{Brown}$\gamma/\rho$ at 250 220
\pinlabel \color{black}{$/\rho$} at 190 172

 \endlabellist
    \centering
    \includegraphics[scale=0.7]{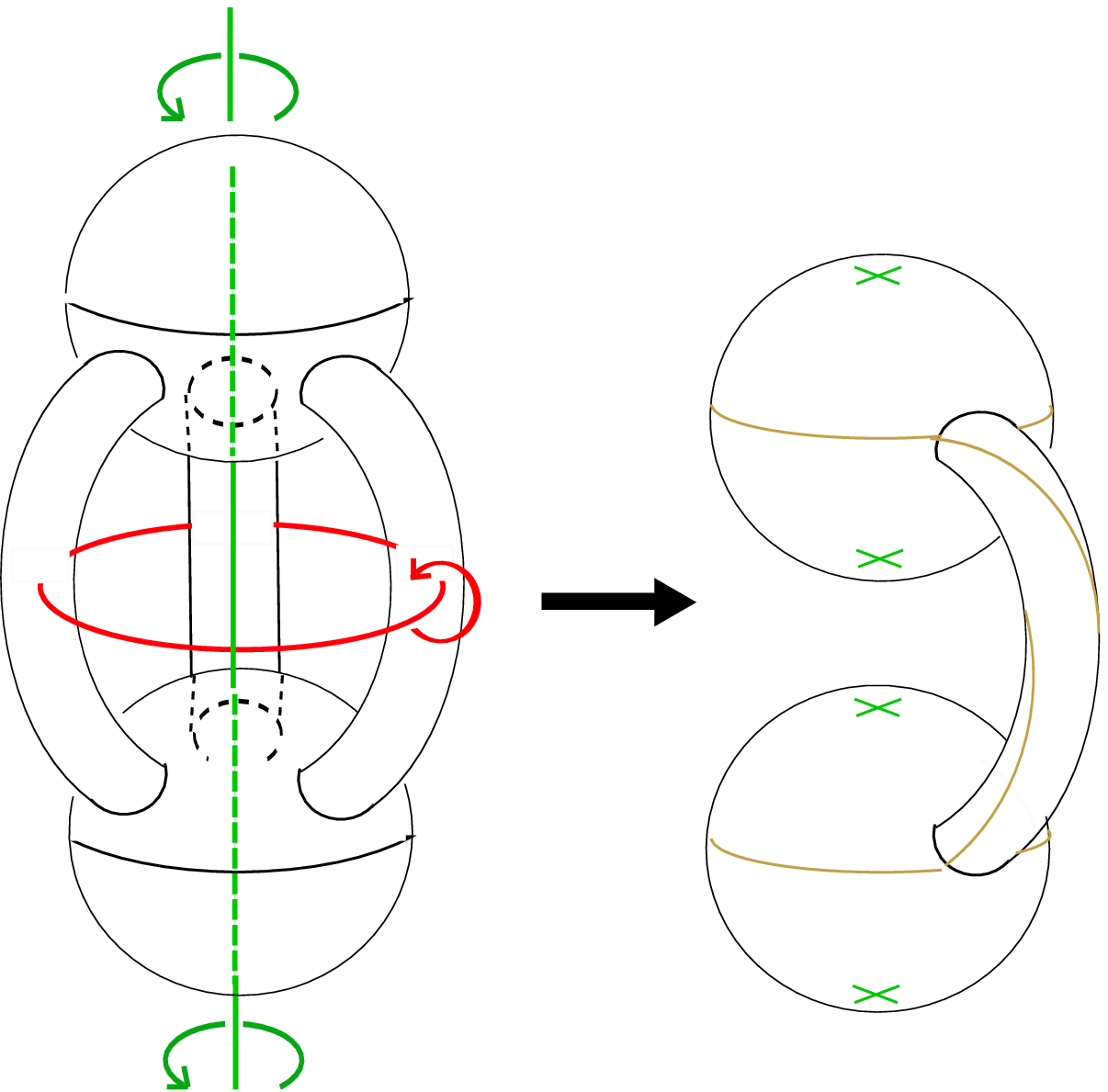}
    \caption{} \label{fig:rhosigma}
    \end{figure}

Regard $F$ as obtained from two spheres by attaching $3$ tubes between them.  See Figure \ref{fig:rhosigma}.  There is an obvious period $3$ symmetry $\rho: F \to F$ gotten by rotating $\frac{2\pi}{3}$ around an axis intersecting each sphere in two points, and a period $2$ symmetry (the hyperelliptic involution) $\sss: F \to F$ obtained by rotating around a circular axis that intersects each tube in two points.  Then $h = \rho \circ \sss = \sss \circ \rho$ is an automorphism of $F$ of order $2 \times 3 = 6$.

The quotient of $F$ under the action of $\rho$ is a sphere with $4$ branch points, each of branching index $3$.  Let $P$ be the $4$-punctured sphere obtained by removing the branch points.   A simple closed curve in $P$ is {\em essential} if it doesn't bound a disk and is not $\bdd$-parallel.  Put another way, a simple closed curve in $P$ is essential if and only if it divides $P$ into two twice-punctured disks.
It is easy to see that there is a separating simple closed curve $\gamma \subset F$ that is invariant under $\sss$ and $\rho$, and hence under $h$, that separates $F$ into two punctured tori $F_R$ and $F_L$; the restriction of $h$ to $F_R$ or $F_L$ is the monodromy of the trefoil knot.   The quotient of $\gamma$ under $\rho$ is shown as the brown curve in Figure \ref{fig:rhosigma} .

(Digressing a bit, it is instructive to compare in this example the monodromy of the closed surface $F$ with that of the Seifert surface $h_-: F_- \to F_-$ appearing in Proposition \ref{prop:strictmono}.  To recover $h_-$ from the given description of $h: F \to F$, consider $F_-$ as the complement of a point $* \in \gamma$ in $F$.  The image of the point $*$ under $h$ is another point in $\gamma$; to see the monodromy on $F_- = F - \{*\}$, compose $h$ with an isotopy of $F$ that slides $h(*)$ back to $*$ along a subarc of $\gamma$.   Typically, even a simple closed curve in $F$ that is carried by $h$ to a disjoint simple closed curve will intersect its image once $h(*)$ is slid back to $*$.)

An important property of the hyperelliptic involution $\sss: F \to F$ is that it fixes the isotopy class of each curve.  That is, for any simple closed curve $c \subset F$, $\sss(c)$ is isotopic to $c$.  The isotopy preserves orientation on $c$ if and only if $c$ is separating.

This immediately gives
\begin{cor} For any simple closed curve $c \subset F$, $h(c)$ is isotopic in $F$ to a curve disjoint from $c$ if and only if $\rho(c)$ is isotopic in $F$ to a curve disjoint from $c$.
\end{cor}
\begin{proof}  Since $h^3(c) = \sss(c)$ is isotopic to $c$ in $F$, $h(c)$ is isotopic to a curve disjoint from $c$ if and only if $h(c)$ is isotopic to a curve disjoint from $h^3(c)$.  Applying $h^{-3}: F \to F$ shows that $h(c)$ is isotopic to a curve disjoint from $h^3(c)$ if and only if $h^{-2}(c) = \rho(c)$ is isotopic to a curve disjoint from $c$.
\end{proof}

  \begin{lemma} \label{lemma:overline}For the $3$-fold branched covering $/\rho: F \to S^2 \supset P$ described above: \begin{enumerate}
 \item An essential simple closed curve $c \subset F$ has the property that $\rho(c)$ can be isotoped to be disjoint from $c$ if and only if $c$ is isotopic to a lift of an essential simple closed curve $\overline{c}$ in $P$.

 \item  A lift $c$ of an essential simple closed curve $\overline{c}$ in $P$ is separating in $F$ if and only if $\overline{c}$ separates the pair of branch points coming from $F_L$ from the branch points coming from $F_R$.

 \item  A single lift $c$ of an essential simple closed curve $\overline{c}$ in $P$ is non-separating in $F$ if and only if $c$ projects homeomorphically to $\overline{c}$.
 \end{enumerate}
 \end{lemma}

 \begin{proof}  Since $\overline{c}$ is embedded, any lift $c$ of $\overline{c}$ will either be carried by $\rho$ to a disjoint curve, or to $c$ itself.  In either case, $\rho(c)$ can be isotoped off of $c$.  On the other hand, suppose $\rho(c)$ can be isotoped off of $c$.  Give $F$ the standard hyperbolic metric, which is invariant under $\rho$, and isotope $c$ to a geodesic in $F$.  Then $\rho(c)$ is also a geodesic and, since they can be isotoped apart, either $\rho(c) = c$ or $c \cap \rho(c) = \emptyset$.  Apply $\rho^{\pm 1}$ to deduce that  $c, \rho(c), \rho^{2}(c)$ either all coincide or are all disjoint.  In either case, $c$ projects to an embedded curve in $P$.  This proves the first part of the Lemma.

To prove the second and third parts, we first establish:

\bigskip

{\bf Claim:}  An essential simple closed curve $\overline{c} \subset P$ lifts to a curve $c$ that is invariant under $\rho$ if and only if $c$ is separating in $F$.

\bigskip

\noindent {\em Proof of Claim.}  Since $\pm 1$ are not eigenvalues of the monodromy matrix for $h_*: \mathbb{Z}^4 \to \mathbb{Z}^4$, no non-separating curve is left invariant by $h$ or $\rho$.  On the other hand, if the essential curve $c$ (and hence also $\rho(c)$) is separating and $\rho(c)$ is disjoint from $c$ then, since $F$ only has genus $2$, $c$ and $\rho(c)$ must be parallel in $F$.  This proves the claim.

\bigskip

The Claim establishes the third part of the Lemma.  For the second part, we restate the Claim: whether $\overline{c}$ lifts to a non-separating or to a separating curve is completely determined by whether $\overline{c}$ is covered in $F$ by three distinct curves or is thrice covered by a single curve.  That, in turn, is determined by whether $\overline{c}$ represents an element of $\pi_1(P)$ that is mapped trivially or non-trivially to $\mathbb{Z}_3$ under the homomorphism $\pi_1(P) \to \mathbb{Z}_3$ that defines the branched covering $F \to S^2 \supset P$.  One way to view such a curve $\overline{c}$ in $P$ is as the boundary of a regular neighborhood of an arc $\underline{c}$ between two branch points of $S^2 \supset P$, namely a pair of branch points (either pair will do) that lie on the same side of $\overline{c}$ in $S^2$.  From that point of view, $\overline{c}$ represents a trivial element of $\mathbb{Z}_3$ if and only if the normal orientations of the branch points at the ends of the arc $\underline{c}$ disagree; that is, one is consistent with a fixed orientation of the rotation axis and one is inconsistent.  Put another way, $\overline{c}$ represents a trivial element if and only if one endpoint of $\underline{c}$ is a north pole in Figure \ref{fig:rhosigma} and the other is a south pole.  The brown curve in the Figure, which lifts to the curve separating $F_R$ from $F_L$, separates both north poles from both south poles.  This establishes the second part of the Lemma.  \end{proof}

\begin{lemma} \label{lemma:framing} If $c \subset F_- \subset F$ projects to a simple closed curve $\overline{c}$ in $P$ then the framing of $c$ given by the Seifert surface $F_- \subset S^3$ of $Q$ is the $0$-framing. \end{lemma}

\begin{proof}  Here is an equivalent conclusion:  If $c'$ is a parallel copy of $c$ in $F_-$, then $\lambda_F(c) \equiv link(c, c') = 0$ in $S^3$.  This self-linking number $\lambda_F(c)$ is independent of the orientation of $c$ (since reversing orientation reverses also the orientation of its push-off $c'$) and doesn't change when $c$ is isotoped in $F_-$.  One way of viewing $\lambda_F(c)$ is this: take an annular neighborhood of $c$ in $F_-$ and ask whether in $S^3$ the annulus is right-hand twisted ($\lambda_F(c) > 0$), left-hand twisted ($\lambda_F(c) < 0$) or not twisted at all ($\lambda_F(c) = 0$).  Here are some useful properties:

\begin{enumerate}[ (a) ]
\item Let $b$ be a simple closed curve in $F_-$ that is parallel to $\bdd F_-$.  Then $\lambda_F(b) = 0$.

Use as a representative annular neighborhood of $b$ the annulus between $b$ and $\bdd F_-$.  Then the complement of that annulus in $F_-$ provides a null-cobordism of one end of the annulus that is disjoint from the other end of the annulus, hence the two ends of the annulus have trivial linking number.

\item  Suppose $c$ is a simple closed curve in $F_-$ and $c'$ is a curve obtained by band-summing $c$ to the curve $b$ just described, via a band that lies in $F_-$.  Then  $\lambda_F(c) = \lambda_F(c').$

First note that a slight push-off of $F_-$ from itself gives a null-cobordism of $b$ that is disjoint from $c$, so $link(b, c) = 0$.  Now a direct argument in a spirit similar to that of the framing addition formula for handle-slides (see \cite[5.1.1]{GS}) shows $\lambda_F(c') = \lambda_F(c) + \lambda_F(b) + 2link(b, c) = \lambda_F(c).$

\item \label{enum:isotopy} Suppose $c_1, c_2$ are simple closed curves in $F_-$ that are isotopic {\em in $F$} (not necesarily in $F_-$).  Then $\lambda_F(c_1) = \lambda_F(c_2).$  Thus $\lambda_F$ can be extended unambiguously to isotopy classes of curves in $F$.

An isotopy in $F$ can be broken up into a series of isotopies in $F_-$ and (when the puncture in $F$ corresponding to $\bdd F_-$ is crossed) band sums to the curve $b$ described above.  We have just seen that both moves leave $\lambda_F$ unchanged.

\item \label{enum:sigma} For any simple closed curve $c$ in $F$ and $\sss: F \to F$ the hyperelliptic involution, $\sss(c)$ is isotopic to $c$ in $F$, so $\lambda_F(\sss(c)) = \lambda_F(c).$

\item \label{enum:rho} For any simple closed curve $c$ in $F$ and $\rho: F \to F$ the period $3$ symmetry described above, $\lambda_F(\rho(c)) = \lambda_F(c).$

The monodromy $h:F_- \to F_-$ doesn't change $\lambda_F$, since the fiber structure gives an isotopy in $S^3$ of any annulus in $F_-$ to its image under the monodromy.  According to (\ref{enum:sigma}) $\sss$ also doesn't change $\lambda_F$, so then neither does $h \circ \sss = \rho$.

\end{enumerate}

Now consider the following automorphism $\tau$ of $S^3$ that preserves the Seifert surface $F_-$ of $Q$, as viewed in the left side of Figure  \ref{fig:tau}:  Reflect through the plane of the knot projection (this reflection makes the right-hand trefoil a left-hand trefoil, and vice versa), then do a $\pi$ rotation around an axis in $S^3$ that passes through the center of the arc $\gamma_- \subset F_-$ that separated $K_R$ from $K_L$.  The automorphism $\tau$ is orientation preserving on the plane of reflection, since the reflection is the identity on this plane and the rotation is orientation preserving.  Hence $\tau$ is orientation preserving also on the Seifert surface, since, for example, we can suppose the Seifert surface has a point of tangency with the plane of reflection throughout the construction.  On the other hand, $\tau$ is orientation reversing on $S^3$.  Because of the orientation reversal of $S^3$ we deduce that for any simple closed curve $c$ in $F_-$, $\lambda_F(\tau(c)) = -\lambda_F(c)$ or, more concretely, $\tau$ changes a right-hand twist in the annulus about $c$ in $F$ into a left-hand twist and vice versa.  A visual example of this phenomenon is how $\tau$ carries the right-hand trefoil summand, with right-hand twists in its part of the Seifert surface, to the left-hand trefoil summand with left-hand twists in its part of the Seifert surface.

The automorphism $\tau |F_-$ descends to an automorphism $\overline{\tau}$ of $S^2 \supset P$ given by $\pi$-rotation about the blue axis shown in Figure  \ref{fig:tau}.  This rotation exchanges $F_R$ and $F_L$ but, up to isotopy and orientation, leaves invariant the simple closed curves in $P$.  In particular, if $c \subset F$ projects to a simple closed curve $\overline{c} \subset P$, then $\tau(c)$ projects to $\overline{\tau}(\overline{c})$ which is isotopic to $\overline{c}$ in $P$.  The isotopy lifts to $F$, so it follows that $\tau(c)$ is isotopic in $F$ to some $\rho^i(c), 0 \leq i  \leq 2$.  Hence by (\ref{enum:isotopy}) and (\ref{enum:rho}) above  $$\lambda_F(c) = -\lambda_F(\tau(c)) = -\lambda_F(\rho^i(c)) = -\lambda_F(c).$$  This implies that $\lambda_F(c) = 0$.
\end{proof}

 \begin{figure}[ht!]
  \labellist
\small\hair 2pt
\pinlabel \color{blue}{$\overline{\tau}$} at 285 125
\pinlabel \color{black}{$\gamma_-$} at 190 90
 \endlabellist
    \centering
    \includegraphics[scale=0.8]{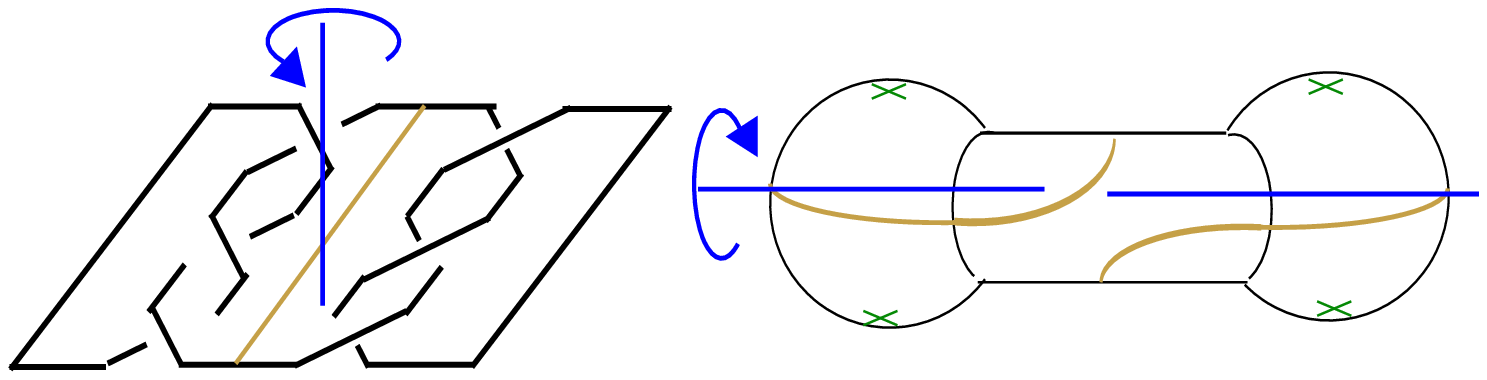}
    \caption{} \label{fig:tau}
    \end{figure}

\begin{cor} \label{cor:enumerate} Suppose $Q \subset S^3$ is the square knot with fiber $F_- \subset S^3$ and $V \subset S^3$ is a disjoint knot.  Then $0$-framed surgery on $Q \cup V$ gives $\#_{2} (S^{1} \times S^{2})$ if and only if, after perhaps some handle-slides of $V$ over $Q$, $V$ lies in $F_-$ and the quotient map $(/\rho):F \to S^2$ projects $V$ homeomorphically to an essential simple closed curve in $P$.
\end{cor}

Note that, according to Lemma \ref{lemma:overline}, essential simple closed curves $\overline{c}$ in $P$ that are such homeomorphic projections are precisely those for which one branch point of $F_L$ (or, equivalently, one branch point from $F_R$) lies on each side of $\overline{c}$. So another way of saying that $(/\rho)$ projects $V$ homeomorphically to an essential simple closed curve in $P$ is to say that $V$ is the lift of an essential simple closed curve in $P$ that separates one branch point of $F_L$ (or, equivalently $F_R$) from the other.

\begin{proof}  This is an application of Proposition \ref{cor:genustwo}.  Suppose that  $0$-framed surgery on $Q \cup V$ gives $\#_{2} (S^{1} \times S^{2})$.  Then, after possible handle-slides over $Q$, $V$ lies in $F_-$.
According to Proposition \ref{cor:genustwo}, in the closed surface $F$, $h(V)$ can be isotoped to be disjoint from $V$ but is not isotopic to $V$.  The result then follows from Lemma \ref{lemma:overline}.

Conversely, suppose that after some handle-slides of $V$ over $Q$, $V$ lies in $F_-$ and $(/\rho)$ projects $V$ homeomorphically to an essential simple closed curve $\overline{c}$ in $P \subset S^2$.  $\rho(V)$ then can't be isotopic to $V$ in $F$ (or else in fact $\rho(V) = V$ and so $(/\rho)|V: V \to \overline{c}$ would not be a homeomorphism) but $\rho(V)$ is disjoint from $V$, since $\overline{c}$ is embedded.  Thus the second and third  conditions of Proposition \ref{cor:genustwo} are satisfied.  Lemma \ref{lemma:framing} establishes the fourth condition.
\end{proof}

\section{The $4$-manifold viewpoint: a non-standard handle structure on $S^4$}    \label{sect:nonstandard}

To understand why the square knot $Q$ is likely to fail Property~2R, it is useful to ascend to 4 dimensions. We use an unusual family of handlebodies from \cite{Go1} to construct a family of 2-component links in $S^3$, each containing $Q$ as one component. Each of these links will fail to satisfy Generalized Property R if a certain associated presentation of the trivial group fails (as seems likely) to satisfy the Andrews-Curtis Conjecture. These links will also be the centerpiece of our discussion of the Slice-Ribbon Conjecture in Section~\ref{slice}.

In \cite{Go1}, the first author provided unexpected examples of handle structures on homotopy $4$-spheres which do not obviously simplify to give the trivial handle structure on $S^4$.  At least one family is highly relevant to the discussion above.  This is example \cite[Figure 1]{Go1}, reproduced here as the left side of Figure \ref{fig:Gompffig1b}.  (Setting $k = 1$ gives rise to the square knot.)  The two circles with dots on them represent $1$-handles, indicating that the $4$-manifold with which we begin is $(S^1 \times D^3) \natural (S^1 \times D^3)$ with boundary $\#_{2} (S^{1} \times S^{2})$ given by $0$-surgery.  (Think of the two dotted unknotted circles as bounding disjoint unknotted disks in the $4$-ball; the dots indicate one should scoop these disks out of $D^4$ rather than attach $2$-handles to $D^4$.)  The circles without dots represent $2$-handles attached to $\#_{2} (S^{1} \times S^{2})$ with the indicated framing.  A sequence of Kirby operations in \cite[\S 2]{Go1} shows that the resulting $4$-manifold has boundary $S^3$.

 \begin{figure}[ht!]
  \labellist
\small\hair 2pt
\pinlabel {\tiny $-n-1$} at 156 758
\pinlabel {\tiny $-n-1$} at 451 758
\pinlabel $n$ at 154 678
\pinlabel $n$ at 448 678
\pinlabel $0$ at 177 718
\pinlabel $-1$ at 117 718
\pinlabel $k$ at 85 718
\pinlabel $-k$ at 232 718
\pinlabel $k$ at 377 718
\pinlabel $-k$ at 527 718
\pinlabel $[0]$ at 470 718
\pinlabel $[0]$ at 336 772
\pinlabel $[0]$ at 336 666
\pinlabel $[-1]$ at 422 705
\pinlabel $0$ at 422 733
\pinlabel $0$ at 508 635
\endlabellist
    \centering
    \includegraphics[scale=0.7]{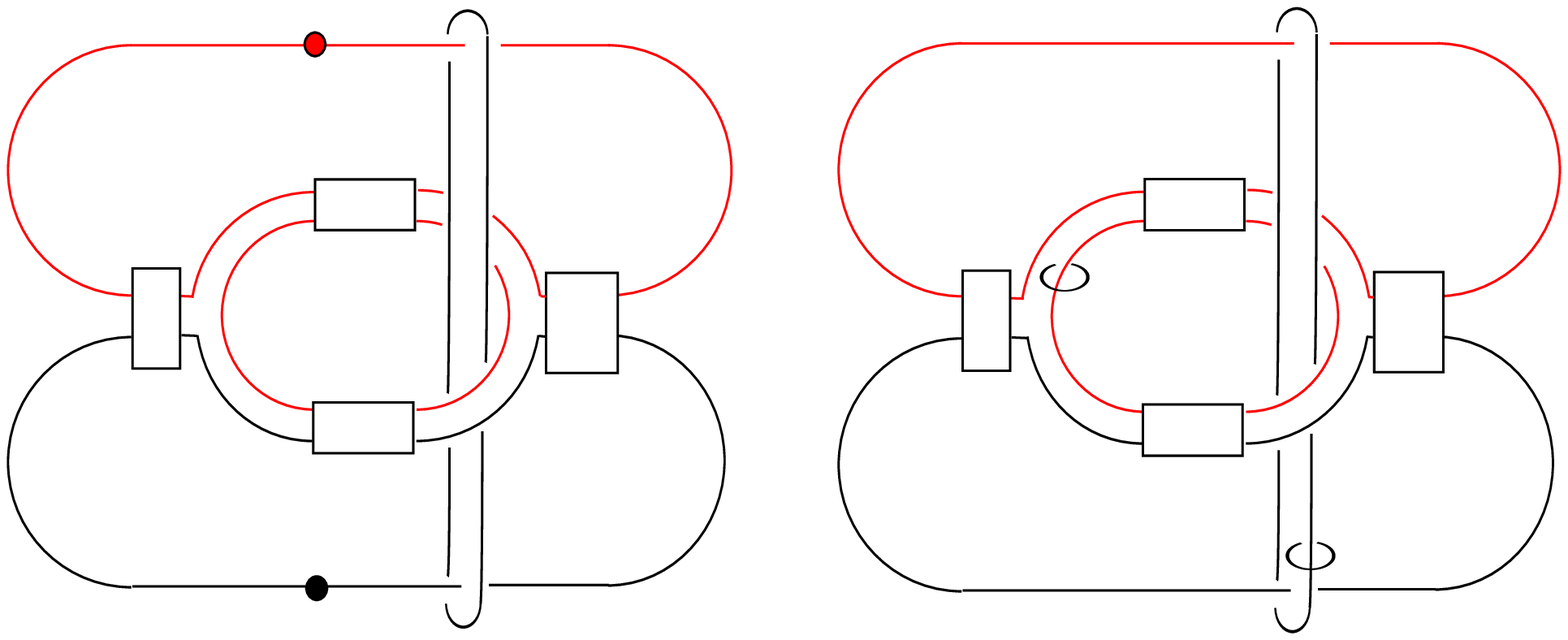}
    \caption{} \label{fig:Gompffig1b}
    \end{figure}

We will be interested in the $4$-manifold that is the trace of the $2$-handle surgeries, the manifold that lies between $\#_{2} (S^{1} \times S^{2})$ and $S^3$.  Our viewpoint, though, is dual to that in \cite{Go1}; the $4$-manifold is thought of as starting with $S^3$ (actually, with $\bdd D^4$) to which two $2$-handles are attached along a link $L_{n,k}$ to get $\#_{2} (S^{1} \times S^{2})$.  This dual viewpoint puts the construction solidly in the context of this paper.

A recipe for making the switch in perspective is given in \cite[Example 5.5.5]{GS}.  The result is shown in the right half of Figure \ref{fig:Gompffig1b}.  The circles representing $1$-handles are changed to $0$-framed $2$-handles and then all $2$-handles have their framing integers bracketed.  We call these circles bracketed circles; they form a complicated surgery description of $S^3$.  New $0$-framed $2$-handles are added, each linking one of these original $2$-handles.  These small linking $2$-handles will eventually be the link $L_{n,k}$, after we eliminate the other components by following the recipe for turning the original diagram into a trivial diagram of the $3$-sphere.  This process allows  handle slides of bracketed or unbracketed circles over bracketed ones, but does not allow handle slides of bracketed circles over unbracketed circles.  The reduction is done in Figure \ref{fig:Gompffig2b} and roughly mimics  \cite[Figures 7 and 8]{Go1}: The top (red) circle is handle-slid over the bottom (black) circle $C$ via a band that follows the tall thin $[0]$-framed circle.  In order to band-sum to a $0$-framed push-off  of $C$, the $n$ twists that are added in the lower twist box are undone by adding $-n$ twists via a new twist box to the left.  The Hopf link at the bottom of the left diagram is removed in a $2$-stage process: The unbracketed component is handleslid over $C$, so it is no longer linked with the $[0]$-framed component of the Hopf link.  Then $C$ is canceled with the $[0]$-framed circle of the Hopf link.  The result is the picture on the right.  The apparent change is both that the Hopf link has been removed and the framing of $C$ has been changed from $[0]$ to $0$.

 \begin{figure}[ht!]
  \labellist
\small\hair 2pt
\pinlabel {\small $-n-1$} at 150 520
\pinlabel {\small $-n-1$} at 450 520
\pinlabel $n$ at 455 440
\pinlabel $n$ at 150 440
\pinlabel $-n$ at 25 425
\pinlabel $-n$ at 337 425
\pinlabel $0$ at 176 360
\pinlabel $[0]$ at 25 476
\pinlabel $C$ at 275 476
\pinlabel $0$ at 340 476
\pinlabel $[0]$ at 140 360
\pinlabel $[0]$ at 40 530
\pinlabel $[0]$ at 352 530
\pinlabel $[-1]$ at 125 500
\pinlabel $[-1]$ at 433 500
\pinlabel $0$ at 180 495
\pinlabel $0$ at 480 500
\pinlabel $k$ at 78 476
\pinlabel $-k$ at 227 476
\pinlabel $k$ at 380 476
\pinlabel $-k$ at 527 476

\endlabellist
    \centering
    \includegraphics[scale=0.7]{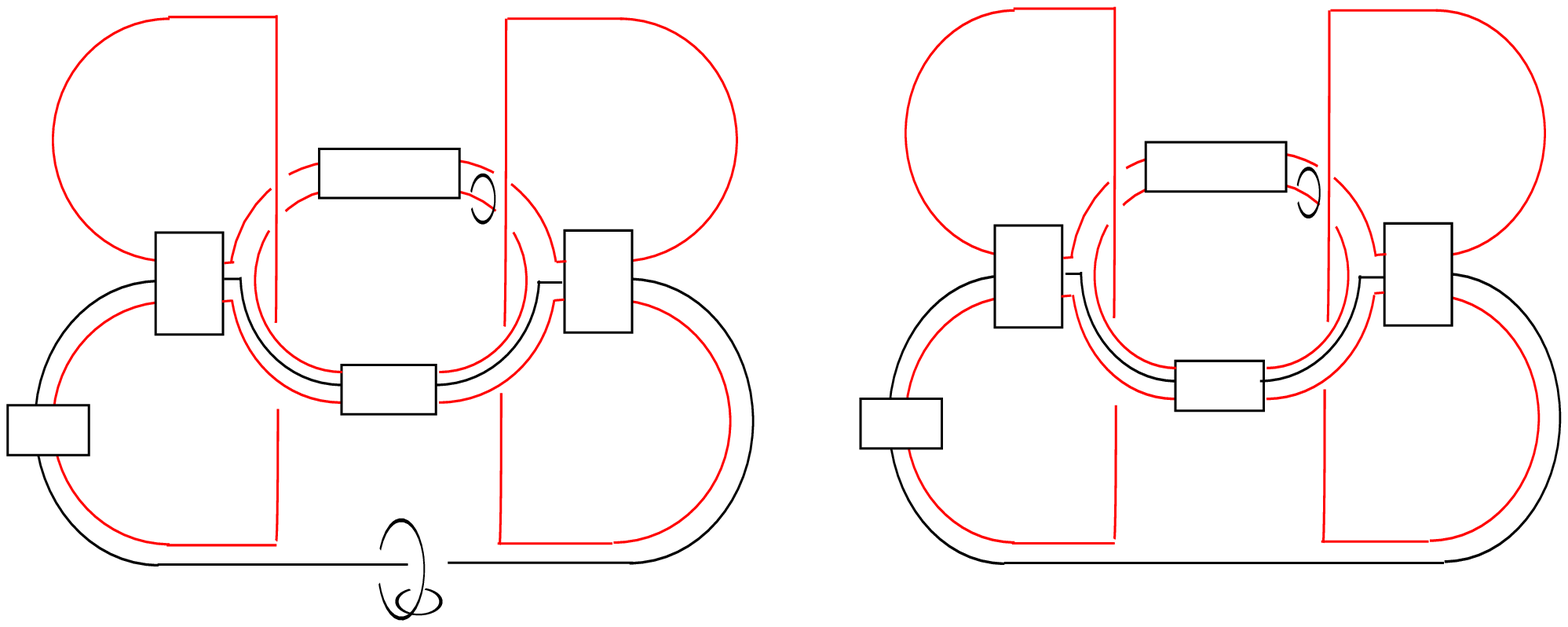}
    \caption{} \label{fig:Gompffig2b}
    \end{figure}

Next (Figure \ref{fig:Gompffig3b}) the large red $[0]$-framed circle is handle-slid over the central red $[-1]$ -framed circle.  That gives the figure on the left.  Note that the twist boxes guarantee we have pushed off with framing $-1$ as required.  Since there is now only one strand going through the top twist-box, the box can be ignored.  The handle-slide changes the $[0]$-framing on the slid circle to a $[+1]$-framing, since the two curves, if oriented in the same direction through the twist boxes, have linking number $n + (-n-1) = -1$ and, with that orientation, the handle-slide corresponds to a handle subtraction. To get the figure on the right, the lower black circle is moved into position and a flype moves the lower twist-box to a new upper twist-box.  (Note the sphere enclosing the left half of the diagram, intersecting the $n$-twist box and two other strands.)  If we ignore the black circles, the red curves form an unlink. Blowing these down would yield $S^3$ with the black curves realizing the desired link $L_{n,k}$.

 \begin{figure}[ht!]
  \labellist
\small\hair 2pt
\pinlabel $n$ at 155 673 %
\pinlabel $n$ at 435 730 %
\pinlabel $-n$ at 35 656 %
\pinlabel $-n$ at 310 650 %
\pinlabel $0$ at 40 705 %
\pinlabel $0$ at 320 705 %
\pinlabel $[1]$ at 50 755 %
\pinlabel $[1]$ at 325 755 %
\pinlabel $[-1]$ at 158 740 %
\pinlabel $[-1]$ at 435 690 %
\pinlabel $k$ at 85 705
\pinlabel $-k$ at 230 705
\pinlabel $k$ at 360 705
\pinlabel $-k$ at 509 705
\endlabellist
    \centering
    \includegraphics[scale=0.7]{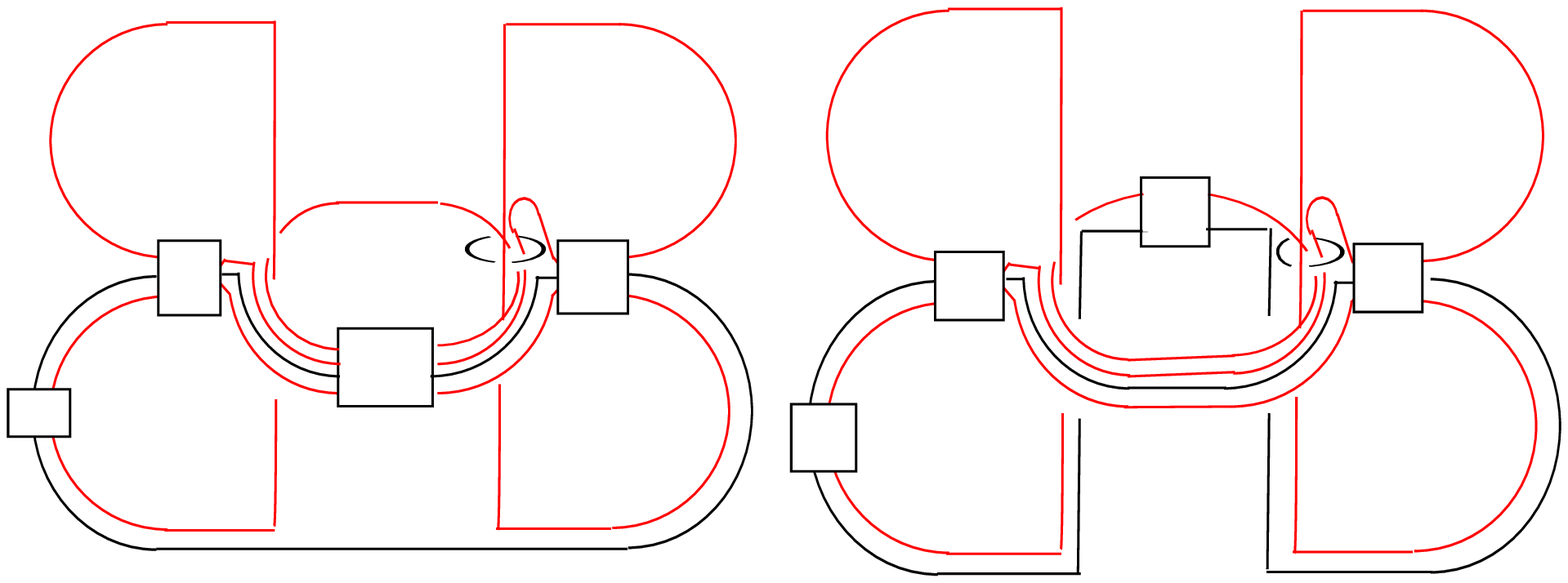}
    \caption{} \label{fig:Gompffig3b}
    \end{figure}

Aim now for the square knot, by setting $k = 1$ and changing the full twists to half twists by a flype.  Figure \ref{fig:Gompffig4} then shows a series of isotopies (clockwise around the figure beginning at the upper left) that simplify the red $[\pm 1]$-framed circles, while dragging along the $0$-framed black circles.

 \begin{figure}[ht!]
  \labellist
\small\hair 2pt
\pinlabel $n$ at 150 365 %
\pinlabel $n$ at 435 370 %
\pinlabel $n$ at 445 180 %
\pinlabel $n$ at 200 175 %
\pinlabel $-n$ at 30 410
\pinlabel $-n$ at 320 415
\pinlabel $-n$ at 40 230
\pinlabel $-n$ at 320 235
%\pinlabel $0$ at 176 360
\pinlabel $0$ at 35 460
\pinlabel $0$ at 325 460
\pinlabel $0$ at 325 300
\pinlabel $0$ at 40 280
\pinlabel $[1]$ at 45 510
\pinlabel $[1]$ at 340 510
\pinlabel $[1]$ at 355 270
\pinlabel $[1]$ at 70 260
\pinlabel $[-1]$ at 130 470
\pinlabel $[-1]$ at 455 410
\pinlabel $[-1]$ at 400 240
\pinlabel $[-1]$ at 240 220
\pinlabel $0$ at 160 450
\pinlabel $0$ at 430 400
\pinlabel $0$ at 470 325
\pinlabel $0$ at 220 270
\endlabellist
    \centering
    \includegraphics[scale=0.7]{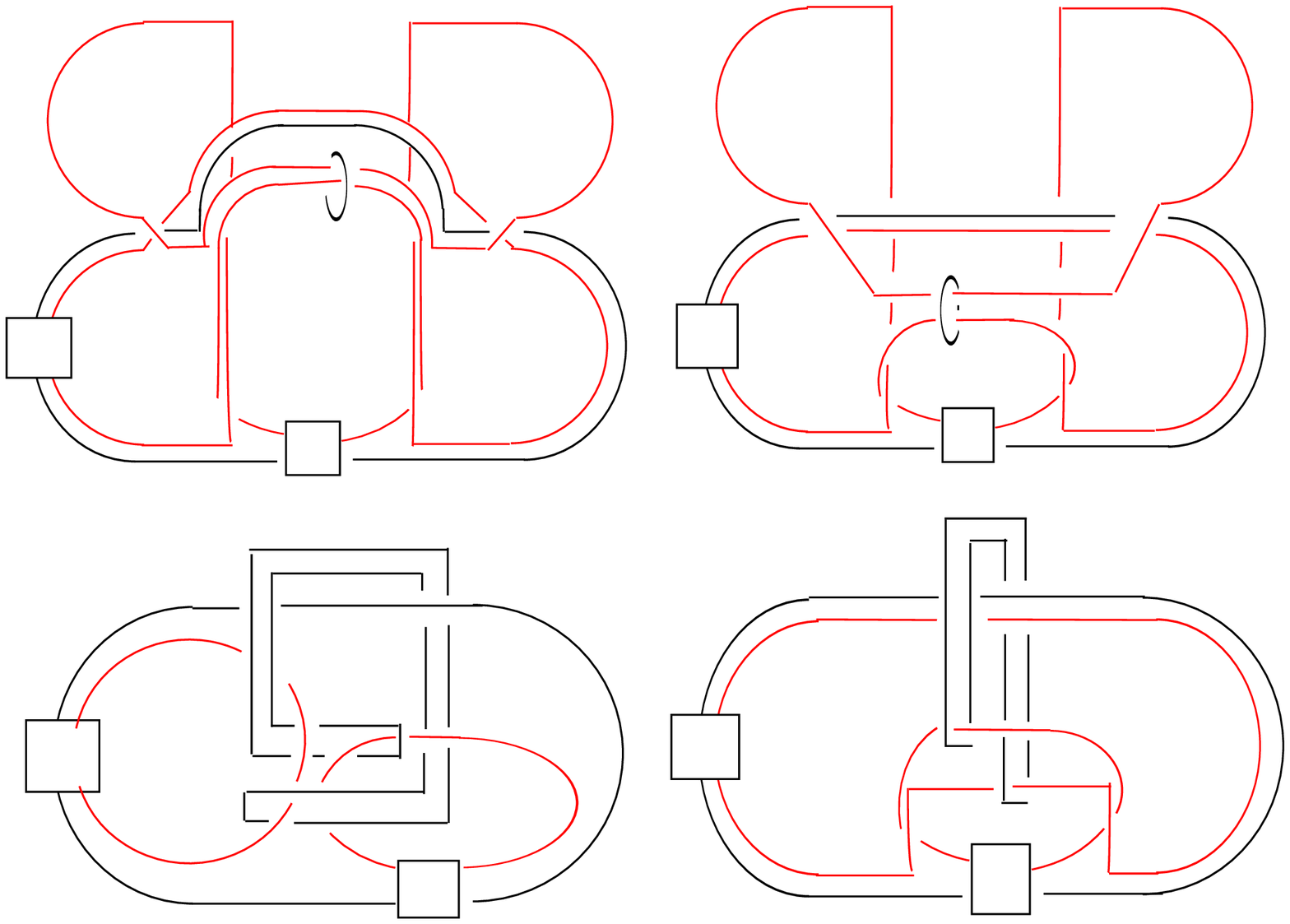}
    \caption{} \label{fig:Gompffig4}
    \end{figure}

Finally, Figure \ref{fig:Gompffig5} shows that the middle $0$-framed component, when the two red $[\pm 1]$-framed components are blown down, becomes the square knot $Q \subset S^3$.  (The other component becomes an interleaved connected sum of two torus knots, $V_n = T_{n, n+1} \# \overline{T_{n, n+1}}$, as shown in Figure~\ref{fig:Ln1}.)

This leaves two natural questions.

     \begin{figure}[ht!]
       \labellist
\small\hair 2pt
     \pinlabel $[1]$ at 75 530
\pinlabel $[-1]$ at 215 460
\endlabellist
    \centering
    \includegraphics[scale=0.7]{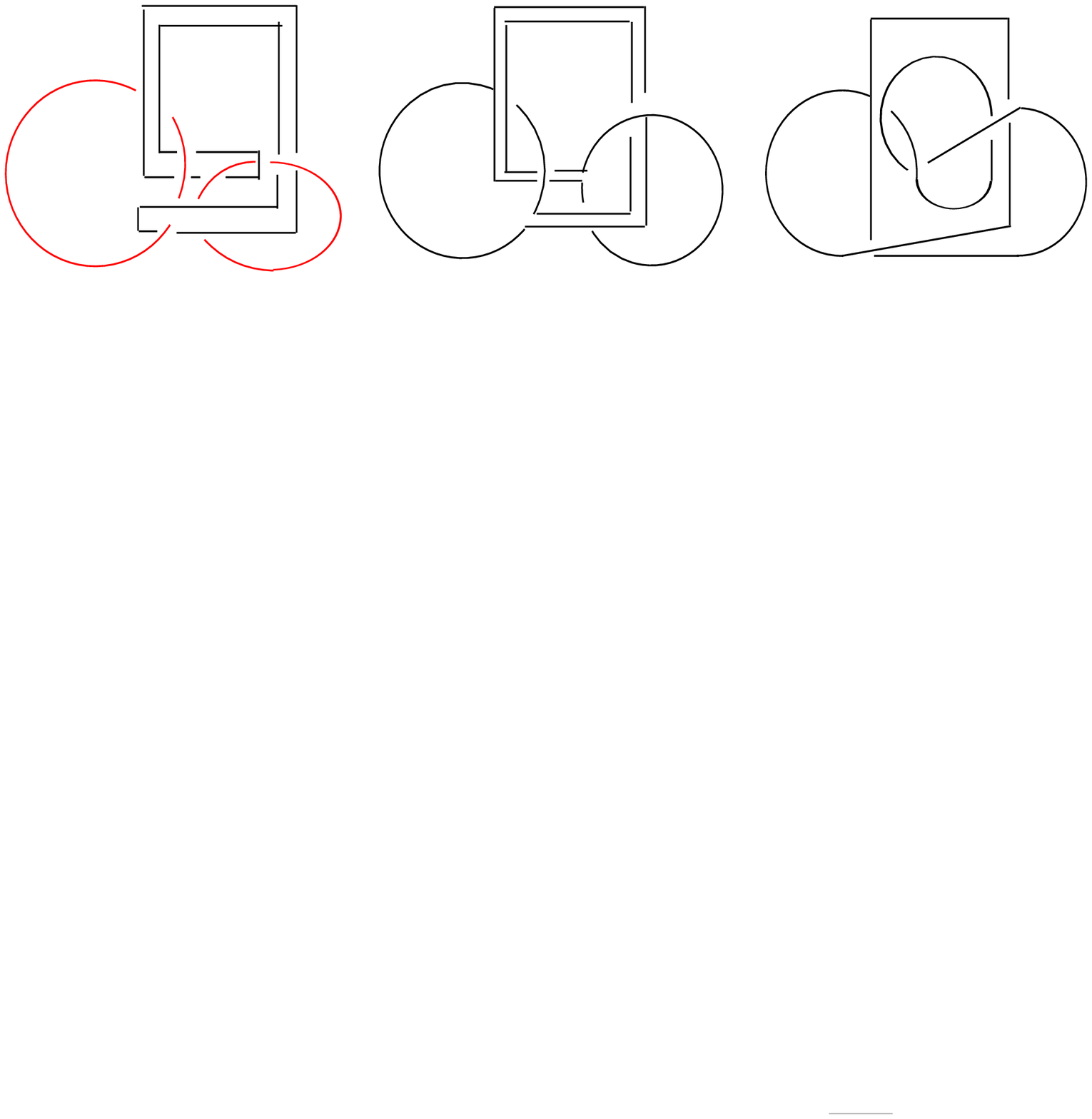}
    \caption{} \label{fig:Gompffig5}
    \end{figure}

\bigskip

{\em Question One:}  As described, $V_n$ does not obviously lie on a Seifert surface for $Q$.  According to Corollary \ref{cor:enumerate}, some handle slides of $V_n$ over $Q$ should alter $V_n$ so that it is one of the easily enumerated curves that do lie on the Seifert surface, in particular it would be among those that are lifts of (half of) the essential simple closed curves in the $4$-punctured sphere $P$.  Which curves in $P$ represent $V_n$ for some $n$?

\bigskip

{\em Question Two:}  Is each $L_{n,1} = Q \cup V_n$, $n \geq 3$, a counter-example to Generalized Property R?

\bigskip

This second question is motivated by Figure~\ref{fig:Gompffig1b}. As described in \cite{Go1}, the first diagram of that figure exhibits a simply connected 2-complex, presenting the trivial group as
 $$\langle x, y \; | \; y = w^{-1} x w,\  x^{n+1} = y^n \rangle,$$
 where $w$ is some word in $x^{\pm 1},y^{\pm 1}$ depending on $k$ and equal to $yx$ when $k=1$. If the 2-component link $L_{n,k}$ of Figure~\ref{fig:Gompffig3b} (obtained by blowing down the two bracketed circles) can be changed to the unlink by handle slides, then the dual slides in Figure~\ref{fig:Gompffig1b} will trivialize that picture, showing that the above presentation is Andrews-Curtis trivial.  For $k=1$, for example, this is regarded as very unlikely when $n \geq 3$. (For $n\le 2$, see Section~\ref{slice}.) Since surgery on $L_{n,k}$ is $\#_2(S^1\times S^2)$ by construction, this suggests an affirmative answer to Question Two, which (for any one $n$) would imply:

\bigskip

\begin{conj} \label{conj:not2R} The square knot does not have Property~2R.
\end{conj}

The invariant we have implicitly invoked here can be described in a purely 3-dimensional way.  Suppose $L$ is an $n$-component framed link that satisfies the hypothesis of Generalized Property~R. Then surgery on $L$ yields $\#_n(S^1\times S^2)$, whose fundamental group $G$ is free on $n$ generators. Its basis $\{g_i\}$ is unique up to Nielsen moves. If we pick a meridian of each component of $L$, attached somehow to the base point, we obtain $n$ elements $\{r_i\}$ that normally generate $G$ (since they normally generate the group of the link complement). Thus, we have a presentation $\langle g_1,\cdots,g_n| r_1,\cdots,r_n\rangle$ of the trivial group. Changing our choices in the construction changes the presentation by Andrews-Curtis moves. If we change $L$ by sliding one component $L_i$ over another $L_j$, then $r_j$ ceases to be a meridian, but it again becomes one via the dual slide over a meridian to $L_i$ (cf. Figure~\ref{fig:dual2}). This is again an Andrews-Curtis move, multiplying $r_j$ by a conjugate of $r_i$. Thus, the link $L$ up to handle slides determines a balanced presentation of the trivial group up to Andrews-Curtis moves.  Unfortunately, there is presently no way to distinguish Andrews-Curtis equivalence classes from each other. When such technology emerges, it should be able to distinguish handle-slide equivalence classes of links satisfying the hypothesis of Generalized Property~R, such as the links $L_{n,k}$.   For a related perspective, see also \cite{CL}.

\bigskip

\section{Slice but not ribbon?} \label{slice}

A link in $S^3=\partial B^4$ is called {\em(smoothly) slice} if it bounds a collection of smooth, disjoint disks in $B^4$. It is called {\em ribbon} if the disks can be chosen so that the radial function on $B^4$ restricts to a Morse function without local maxima on the disks. Both conditions are preserved by (0-framed) handle slides, the net effect being to band-sum one disk with a parallel copy of another. Band-summing link components to create a knot similarly preserves both conditions. Slice knots are somewhat rare, and their constructions typically generate slice disks that are easily seen to be ribbon. In fact, the Slice-Ribbon Conjecture, asserting that every slice knot should be ribbon, is still open after more than three decades. The corresponding question for links is also still open. Potential counterexamples to the Generalized Property~R Conjecture are a natural source of potential counterexamples to the Slice-Ribbon Conjecture, since the conclusion of the former conjecture immediately implies a link and its band-sums are ribbon, whereas its hypothesis only guarantees that the link and its band-sums are slice in some possibly exotic smoothing of $B^4$ (Proposition~\ref{prop:slice}). In our case, the smoothing is standard, so we can construct slice knots with no obvious approach to proving that they are ribbon. This method appears to be the only currently known source of potential counterexamples to the Slice-Ribbon Conjecture (for knots or links). In light of Akbulut's recent work \cite{Ak}, more complicated potential counterexamples can also be generated by this same method; cf. \cite{FGMW}.

 \begin{figure}[ht!]
 \labellist
\small\hair 2pt
\pinlabel {{\color{red} left twist}} at 170 215
\pinlabel {{\color{red} right}} at 165 105
\pinlabel {{\color{red} twist}} at 165 90
\pinlabel {$k$ pairs} at 95 25
\pinlabel ${-n}$ at 75 125
\pinlabel ${n}$ at 115 70
  \endlabellist
    \centering
    \includegraphics[scale=0.7]{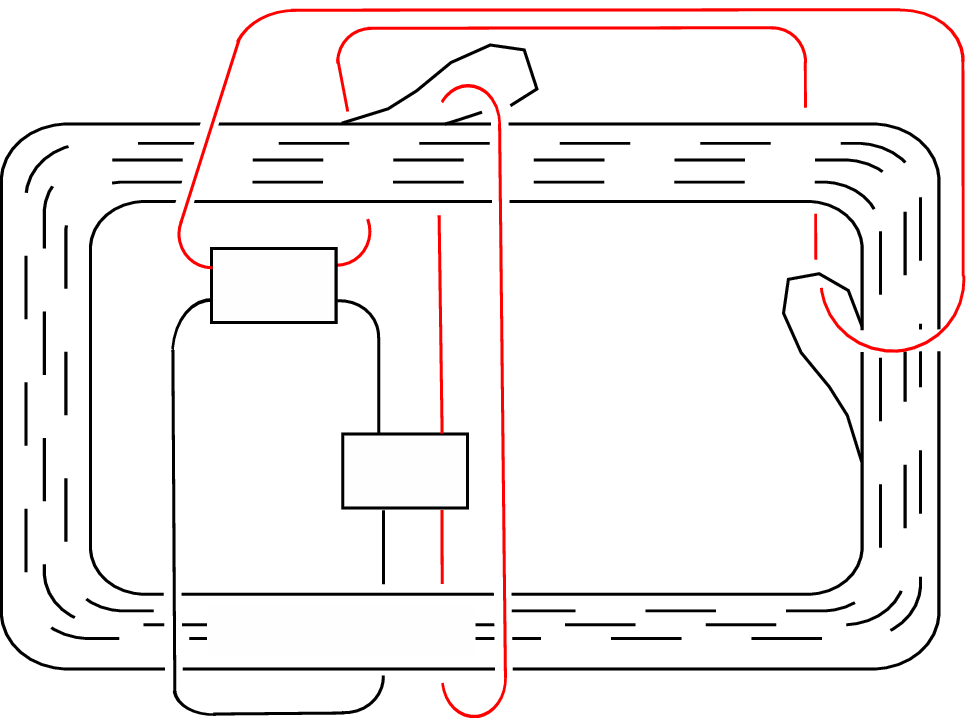}
    \caption{} \label{fig:Lnk}
    \end{figure}

Recall that the link $L_{n,k}$ is obtained from the right side of Figure~\ref{fig:Gompffig3b} by blowing down the red 2-component unlink. This was done explicitly for $k=1$ to obtain Figure~\ref{fig:Ln1}. For general $k$, we at least simplify the red unlink to obtain Figure~\ref{fig:Lnk}. To interpret this diagram, first consider the black unlink in the figure. The large component is the boundary of a disk in the plane that is wrapped in a spiral with $k$ turns. Thus, the circle runs around the spiral with $2k$  strands, occurring in oppositely oriented pairs. To obtain $L_{n,k}$, thread the red curves through the twist boxes as shown, then twist as indicated on the red unlink so that one component becomes the sum of torus knots $T_{n, n+1} \# \overline{T_{n, n+1}}$. (This can be done explicitly, but results in a messier picture.) The link $L_{n,k}$ was first shown in a complicated diagram of $S^3$, Figure~\ref{fig:Gompffig1b}, as the two meridians of the 2-handles. These obviously bound disjoint disks in the pictured 4-manifold, that are the cores of the 2-handles in the dual picture, explicitly exhibiting the proof of Proposition~\ref{prop:slice} in our case. Since the pictured 4-manifold is actually diffeomorphic to a 4-ball \cite{Go1}, it follows that $L_{n,k}$ is smoothly slice.

The authors do not know whether $L_{n,k}$ is ribbon except in the special cases $n=0,1$ or $k=0$ or $(n,k)=(2,1)$. (We continue to assume $n>0$ without loss of generality, since the symmetry of Figure~\ref{fig:Gompffig1b}, $\pi$-rotation about a horizontal axis in the plane of the paper, interchanges $n$ and $-n-1$. However, it may be useful to consider $k<0$.) In each of the special cases, the corresponding presentation is Andrews-Curtis trivial. (This is clear except for the last case, which is due to Gersten \cite{Ge} and reproduced for convenience at the end of this section.) In each case, the algebra can be followed by a trivialization of the handlebody Figure~\ref{fig:Gompffig1b} that introduces no new handles, and the dual computation shows that the corresponding $L_{n,k}$ satisfies the conclusion of Generalized Property~R. Thus, the special cases, and any knots created from them by band-sum, are ribbon. In principle, the ribbon disks can be explicitly determined from the construction, although the details may be laborious. In contrast, the remaining cases are mysterious. In general,  each component of $L_{n,k}$ is ribbon: A band move turns the small component in  Figure~\ref{fig:Gompffig3b} into a meridian of each of the bracketed circles, becoming an unlink when these are blown down. The corresponding move in Figure~\ref{fig:Lnk} is to cut through the spiral disk. The other component is ribbon since it is the sum of a knot with its mirror. Unfortunately, the individual ribbon disks seem to interfere with each other. One could try to trace the given slice disks through the construction, but it seems likely that the 2-3 handle pair needed for canceling the handlebody creates troublesome local maxima. On the other hand, the given slice disks have the unusual property that their complement is a regular neighborhood of a wedge of two circles; perhaps there is a pair of ribbon disks whose complement has more complicated fundamental group. In summary, we have the following:

\bigskip

{\em Questions:}  For $n\ge 2$, $k\ne 0$ and $(n,k)\ne (2,1)$, is $L_{n,k}$ a ribbon link? Are the slice knots made by band-summing its components always ribbon? Are they ever ribbon?

\bigskip

A simple example with $(n,k)= (3,1)$ is given in Figure~\ref{fig:sliceknot}.

Gersten trivializes the presentation $$\langle x,y \; | \; yxyx^{-1}y^{-1}x^{-1},\  x^3y^{-2}\rangle$$ corresponding to
$L_{2,1}$ by multiplying the first relator by $$(x^3y^{-2})[(yxy)y^{-2}x^3(yxy)^{-1}],$$ obtaining (after conjugation) $x^2y^{-1}xy^{-1}x^2y^{-1}$. Replacing the generator $y$ by $z=x^2y^{-1}$ changes the new relator to $x^{-1}z^3$, killing $x$ and hence $z$.

\section{Weakening Property nR} \label{sect:weakening}

The evidence of Section~\ref{sect:nonstandard} suggests that Property~2R, and hence the Generalized Property~R Conjecture, may fail for a knot as simple as the square knot.  There are other reasons to be unhappy with this set of conjectures:   For one thing, there is no clear relation between Property nR and Property (n+1)R:  If $K$ has Property nR, there is no reason to conclude that no $(n+1)$-component counterexample to Generalized Property R contains $K$.  After all, Gabai has shown that {\em every knot has Property 1R}.  Conversely, if $K$ does not have Property nR, so there is an $n$-component counterexample $L$ to Generalized Property R, and $K \subset L$, it may well be that even adding a distant $0$-framed unknot to $L$ creates a link that satisfies Generalized Property R.  So $K$ might still satisfy Property (n+1)R.

A four-dimensional perspective on this last possibility is instructive.  (See also \cite[Section 3]{FGMW}. ) Suppose $L$ is an $n$-component link on which surgery gives $\#_{n} (S^{1} \times S^{2})$.  Suppose further that, after adding a distant $0$-framed $r$-component unlink to $L$, the resulting link $L'$  can be reduced to an $(n+r)$-component unlink by handle-slides.  Consider the closed $4$-manifold $W$ obtained by attaching $2$-handles to $D^4$ via the framed link $L$, then attaching $\natural_{n} (S^{1} \times B^{3})$ to the resulting manifold along their common boundary $\#_{n} (S^{1} \times S^{2})$. Via \cite{LP} we know there is essentially only one way to do this.  The result is a simply-connected (since no $1$-handles are attached) homology $4$-sphere, hence a homotopy $4$-sphere.  If the $2$-handles attached along $L$ can be be slid so that the attaching link is the unlink, this would show that $W \cong S^4$, since it implies that the $2$-handles are exactly canceled by the $3$-handles.  But the same is true if handle-slides convert $L'$ to the unlink: the $4$-manifold $W$ is the same, but first attaching $r$ more trivial $2$-handles and then canceling with $r$ more $3$-handles has no effect on the topology of the associated $4$-manifold.  So from the point of view of $4$-manifolds, this weaker form of Generalized Property R (in which a distant $r$-component unlink may be added to the original link) would suffice.

From the $4$-manifold point of view there is a dual operation which also makes no difference to the topology of the underlying $4$-manifold: adding also pairs of canceling $1$- and $2$- handles.  From the point of view of Kirby calculus, each such pair is conventionally noted in two possible ways (see Section 5.4 of \cite{GS} and Figure \ref{fig:canceling} below):
\bigskip

\begin{itemize}

\item A dumb-bell shaped object.  The ends of the dumb-bell represent $3$-balls on which a $1$-handle is attached; the connecting rod represents the attaching circle for a canceling $2$-handle.

\bigskip

\item A Hopf link.  One component of the link is labeled with a dot and the other is given framing $0$.  The dotted component represents the $1$-handle and the $0$-framed component represents the attaching circle for a canceling $2$-handle.  Call such an object a canceling Hopf pair and call the union of $s$ such pairs, each lying in a disjoint $3$-ball, a {\em set of $s$ canceling Hopf pairs}.
\end{itemize}

\bigskip

(The rules for handle-slides of these dotted components in canceling Hopf pairs are fairly restricted; see \cite{GS}.  For example, they cannot slide over non-dotted components. Note that while we require the 2-handle of a canceling Hopf pair to have framing 0, we can change it to any even framing by sliding the 2-handle over the dotted circle to exploit the nonzero linking number.)

 \begin{figure}[ht!]
  \labellist
\small\hair 2pt
\pinlabel $0$ at 160 650
\pinlabel $0$ at 510 650
 \endlabellist
    \centering
    \includegraphics[scale=0.6]{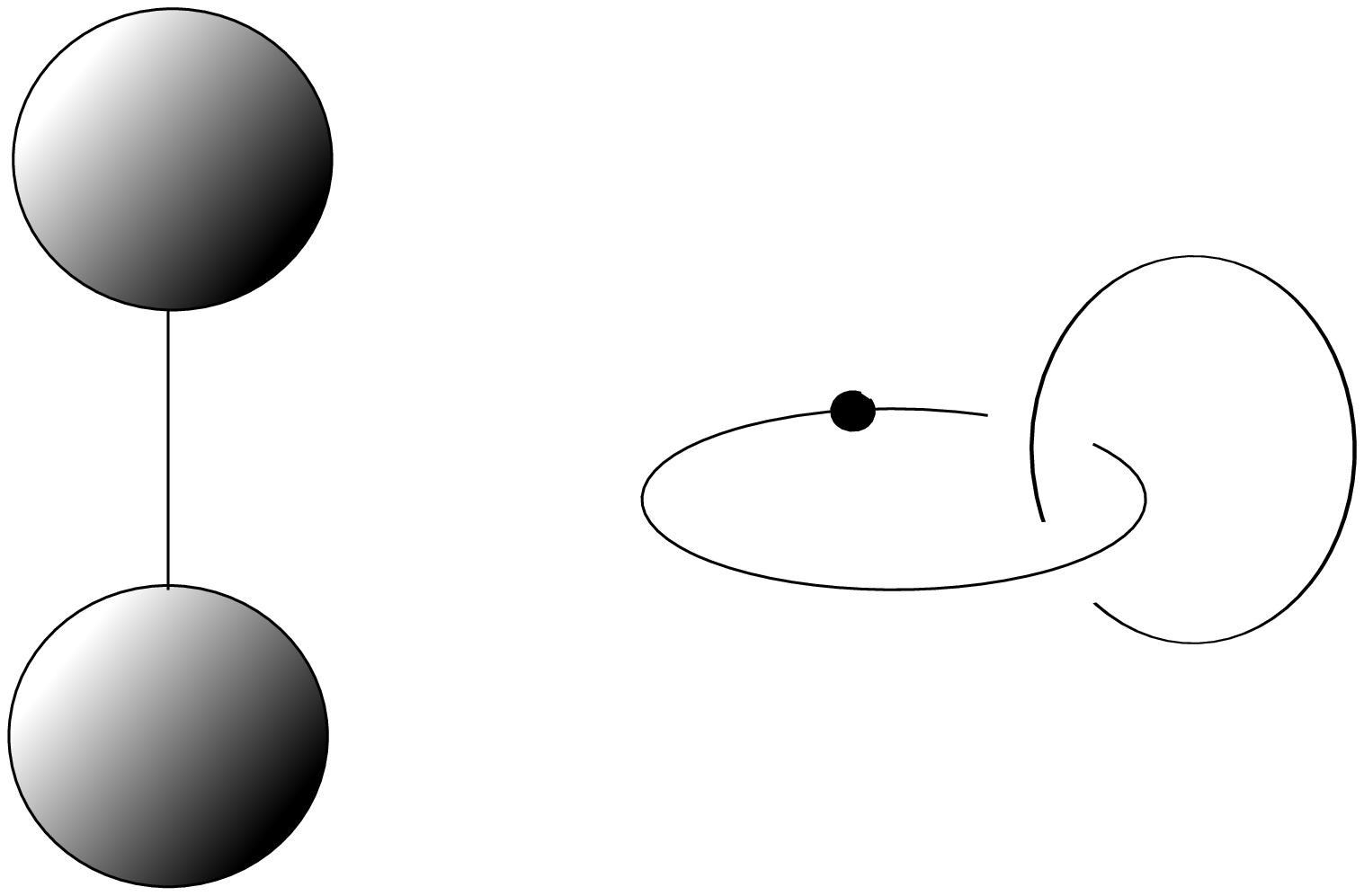}
    \caption{} \label{fig:canceling}
    \end{figure}

The $4$-manifold perspective suggests the following weaker and more awkward version of Generalized Property R:

\begin{conj}[Weak Generalized Property R] \label{conj:weakgenR} Suppose $L$ is a framed link of $n \geq 1$ components in $S^3$, and surgery on $L$ yields $\#_{n} (S^{1} \times S^{2})$.  Then, perhaps after adding a distant $r$-component $0$-framed unlink and a set of $s$ canceling Hopf pairs to $L$, there is a sequence of handle-slides that creates the distant union of an $n+r$ component $0$-framed unlink with a set of $s$ canceling Hopf pairs.
\end{conj}

The move to Weak Generalized Property R destroys the Andrews-Curtis invariant of the previous section, since adding a canceling Hopf pair leaves the surgered manifold, and hence $G$, unchanged, but introduces a new relator that is obviously trivial.  In fact, we will see in the next section that the links $L_{n,k}$ are not counterexamples to this weaker conjecture (even if we set set $r=0$, $s=1$). On the other hand, stabilization by adding a distant unknot to $L$ changes the corresponding presentation by adding a generator $g_{n+1}$ and relator $r_{n+1}=g_{n+1}$, and this preserves the Andrews-Curtis class. Thus, the Hopf pairs seem crucial to the conjecture, whereas it is unclear whether the distant unlink is significant. (Of course, if we allow 1-handles in our original diagram as in Figure~\ref{fig:Gompffig1b}, then the Andrews-Curtis problem seems to make the distant unlink necessary.)

\begin{defin} A knot $K \subset S^3$ has {\bf Weak Property nR} if it does not appear among the components of any $n$-component counterexample to the Weak Generalized Property R conjecture.
\end{defin}

 We have seen that Property nR and Weak Property nR are probably quite different, due to the Hopf pairs. The latter property also exhibits nicer formal behavior due to the distant unlink, in that if $K$ has Weak Property (n+1)R then it has Weak Property nR:  Suppose $K$ appears among the components of an $n$-component framed link $L$ and surgery on $L$ via the specified framing yields $\#_{n} (S^{1} \times S^{2})$. Then the corresponding statement is true for the union $L'$ of $L$ with a distant $0$-framed unknot.  If $K$ has Weak Property (n+1)R then the distant union $L''$ of $L'$ with the $0$-framed $r$-component unlink and a set of $s$ canceling Hopf pairs has a sequence of handle-slides that creates the $0$-framed unlink of $n + r + 1$ components and a set of $s$ canceling Hopf pairs.  But $L''$ can also be viewed as the union of $L$ with the $0$-framed $(r+1)$-component unlink and a set of $s$  canceling Hopf pairs. Thus $K$ satisfies Weak Property~nR.

The Weak Generalized Property R Conjecture is closely related to the Smooth (or PL) 4-Dimensional Poincar\'e Conjecture, that every homotopy 4-sphere is actually diffeomorphic to $S^4$. For a precise statement, we restrict attention to homotopy spheres that admit handle decompositions without 1-handles.

\begin{prop}  The Weak Generalized Property R Conjecture is equivalent to the Smooth 4-Dimensional Poincar\'e Conjecture for homotopy spheres that admit handle decompositions without 1-handles.
\end{prop}

While there are various known ways of constructing potential counterexamples to the Smooth 4-Dimensional Poincar\'e Conjecture, each method is known to produce standard 4-spheres in many special cases. (The most recent developments are \cite{Ak}, \cite{Go2}.) Akbulut's recent work \cite{Ak} has eliminated the only promising potential counterexamples currently known to admit handle decompositions without 1-handles. For 3-dimensional renderings of the full Smooth 4-Dimensional Poincar\'e Conjecture and other related conjectures from 4-manifold theory, see \cite{FGMW}.

\begin{proof}  Suppose $\Sigma$ is a homotopy sphere with no 1-handles and $n$ $2$-handles.  Since there are no $1$-handles, the $2$-handles are attached to some framed $n$-component link $L \subset S^3$ in the boundary of the unique $0$-handle $D^4$ in $\Sigma$. Since $\Sigma$ has Euler characteristic 2, there are $n$ 3-handles attached to the resulting boundary, showing that surgery on $L$ is $\#_n(S^1\times S^2)$.  If the Weak Generalized Property R Conjecture is true, there are some $r, s$ so that when a distant $0$-framed $r$-component unlink and a distant set of $s$ canceling Hopf pairs is added to $L$ (call the result $L'$) then after a series of handle slides, $L'$ becomes the distant union of an $n+r$ component unlink and a set of $s$ canceling Hopf pairs.

To the given handle decomposition of $\Sigma$, add $r$ copies of canceling $2$- and $3$- handle pairs and $s$ copies of canceling $1$- and $2$-handle pairs.  After this change of handle decomposition, the manifold $\Sigma_2$ that is the union of $0$-, $1$- and $2$-handles can be viewed as obtained by the surgery diagram $L' \subset \bdd D^4$.  After a sequence of handle-slides, which preserve the diffeomorphism type of $\Sigma_2$, $L'$ is simplified as above; it can be further simplified by canceling the $1$- and $2$-handle pairs given by the set of $s$ canceling Hopf pairs.  What remains is a handle description of $\Sigma_2$ given by $0$-framed surgery on an unlink of $n+r$ components.

Since $\Sigma$ has Euler characteristic 2, it is obtained by attaching exactly $(n+r)$ $3$-handles and then a single $4$-handle to $\Sigma_2$.  If we view $\Sigma_2$ as obtained by attaching $2$-handles via the $0$-framed $(n+r)$-component unlink, there is an obvious way to attach some set of $(n+r)$ $3$-handles so that they exactly cancel the $2$-handles, creating $S^4$.  It is a theorem of Laudenbach and Poenaru \cite{LP} that, up to handle-slides, there is really only one way to attach $(n+r)$ $3$-handles to $\Sigma_2$.  Hence $\Sigma$ is diffeomorphic to $S^4$ as required.

Conversely, if $0$-framed surgery on an $n$-component link $L \subset S^3$ gives $\#_{n} (S^{1} \times S^{2})$ then a smooth homotopy $4$-sphere $\Sigma$ can be constructed by attaching to the trace of the surgery $n$ $3$-handles and a $4$-handle.  If $\Sigma$ is $S^4$, then standard Cerf theory says one can introduce some canceling $1$- and $2$- handle pairs, plus some canceling $2$- and $3$- handle pairs to the given handle description of $\Sigma$, then slide handles until all handles cancel.  But introducing these canceling handle pairs, when expressed in the language of framed links, constitutes the extra moves that are allowed under Weak Generalized Property R. (To see that the framings in the canceling Hopf pairs can be taken to be 0, first arrange them to be even by choosing the diagram to respect the unique spin structure on $\Sigma$, then slide as necessary to reduce to the 0-framed case.)\end{proof}

\section{How a Hopf pair can help simplify} \label{sect:Hopf}

With the Weak Generalized Property R Conjecture in mind, return now to the square knot example of Section \ref{sect:nonstandard}.   In \cite{Go1} it is shown that the introduction of a canceling $2, 3$-handle pair does make Figure~\ref{fig:Gompffig1b} equivalent by handle-slides to the corresponding canceling diagram.  In our dualized context, that means that the introduction of a canceling $1, 2$ handle pair (in our terminology, a canceling Hopf pair) should allow $Q \cup V_n$ (or more generally any $L_{n,k}$) to be handle-slid until the result is the union of the unlink and a canceling Hopf pair.  In particular, the square knot examples provide no evidence against the Weak Generalized Property R Conjecture.

It is a bit mysterious how adding a Hopf pair can help to simplify links. Since the dotted circle is not allowed to slide over the other link components, it may seem that any slides over its meridian would only need to be undone to remove extraneous linking with the dotted circle. In this section, we give a general procedure that exploits a nontrivial self-isotopy of the meridian to create families of potential counterexamples to Generalized Property~R. The examples $L_{n,k}$ of Section~\ref{sect:nonstandard} arise as a special case for each fixed $k$.  This suggests that Generalized Property~R probably fails frequently for links constructed by this method, whereas Weak Generalized Property~R does not.

\bigskip

Begin with an oriented 3-manifold $M$ containing a framed link $L=L'\cup L''$.  We will construct an infinite family of framed links in $M$ that may not be handle-slide equivalent, but become so after a Hopf pair is added. (In our main examples, $M$ will be $S^3$ with $L$ a 0-framed 2-component link.) Let $M'$ be the manifold made from $M$ by surgery on $L'$.

Let $\varphi: T^2\to M'$ be a generic immersion that is an embedding on each fiber $S^1\times\{t\}$, $t\in S^1=\R/\zed$.  Denote the image circle $\varphi(S^1\times\{t\})$ by $C_t$. Since $\varphi$ is generic, it is transverse both to $L''$ and to the cores of the surgery solid tori in $M'$.  In particular,
\begin{itemize}
\item only finitely many circles $C_{t_i}$  intersect $L''$
\item each such circle $C_{t_i}$ intersects $L''$ in exactly one point
\item the curves $C_{t_i}$ are pairwise disjoint and
\item each $C_{t_i}$ lies in the link complement $M-L' \subset M'$.
\end{itemize}

The immersion $\varphi$ determines an infinite family of framed links $L_n= L'\cup L_n''$ in $M$, $n\in \zed$: For each $t_i$, there is an annulus $A_i$ in $M-L'$ centered on $C_{t_i}$, transverse to $\varphi$ near $S^1\times\{t_i\}$ and intersecting $L''$ in an arc, with $A_i$ oriented so that its positive normal direction corresponds to increasing $t$. See Figure~\ref{fig:Ai}, where the left and right edges of each diagram are glued together. The annuli $A_i$ are unique up to isotopy, and we can assume they are pairwise disjoint. Let $L_n''$ be the framed link obtained from $L''$ by diverting each arc $L''\cap A_i$ so that it spirals $|n|$ times around in $A_i$ (relative to its original position) as in Figure~\ref{fig:Ai}, starting with a right turn if $n>0$. (Thus $L_0=L$.) We specify the framings by taking them to be tangent to $A_i$ everywhere inside these annuli (before and after) and unchanged outside. Note that $L_n$ depends on our initial choices of how to isotope the curves $C_{t_i}$ off of the surgery solid tori, but changing the choices only modifies $L_n$ by handle slides of $L_n''$ over $L'$.

 \begin{figure}[ht!]
 \labellist
\small\hair 2pt
\pinlabel  $t$ at 50 62
\pinlabel  $A_i$ at 100 95
\pinlabel  $A_i$ at 360 95
\pinlabel  {${\rm Im}\; \varphi$} at 85 20
\pinlabel  $C_t$ at 275 80
\pinlabel  $C_{t_i}$ at 245 62
\pinlabel  $L''$ at 145 130
\pinlabel  $L''_2$ at 420 130
\pinlabel  ${\rm framing}$ [l] at 170 130
\endlabellist
     \centering
    \includegraphics[scale=0.7]{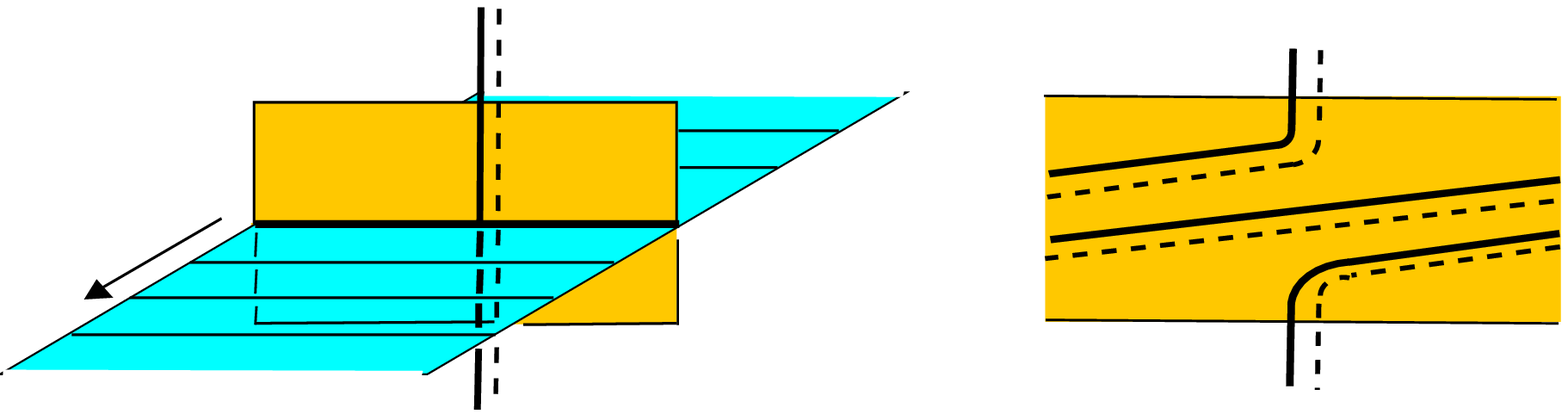}
    \caption{} \label{fig:Ai}
    \end{figure}

In our main examples below, $\varphi$ is an embedding. Whenever this occurs, there is a self-diffeomorphism of $M'$ sending $L''$ onto  $L_n''$, namely the $n$-fold Dehn twist along $\im \varphi$ parallel to the circles $C_t$. This shows directly in this case that the manifolds obtained by surgery on $L_n$ are all diffeomorphic, regardless of $n$. In general, we can obtain more specific information:

\begin{prop} \label{prop:Ln}
For $L\subset M$ and $\varphi$ a generic immersion as above, suppose that $C_0$ is disjoint from all other circles $C_t$  and bounds an embedded disk in $M$ whose algebraic intersection with nearby curves $C_t$ is $\pm 1$. Then all of the framed links $L_n$ become handle-slide equivalent after the addition of a single canceling Hopf pair.
\end{prop}

\noindent Since $C_0$ is nullhomologous in $M$, it has a canonical longitude, determined by any Seifert surface in $M$. Thus, the last hypothesis can be restated by saying that $C_0$ is an unknot in $M$ for which the immersion $\varphi$ determines the $\pm 1$ framing.

To see that the hypotheses are not unreasonable, note that they are satisfied for $M=S^3$, $L'$ empty and $C_0$ a $(1,\pm 1)$-curve in the standard torus $T$. In this case, the links $L''_n$ are obviously all equivalent for any fixed $L''$. (Since $T$ is embedded, the links $L''_n$ are all related as above by Dehn twisting along $T$ in $M'=M=S^3$. These Dehn twists are isotopic to the identity by an isotopy fixing one complementary solid torus of $T$ and twisting the other by both of its circle actions.) The more interesting examples $L_{n,k}$ that arise in \cite{Go1} will be exhibited below. These all arise from the same general procedure: Start with a 0-framed unlink in $S^3$. Slide handles to create a more complicated link $L$ satisfying Generalized Property~R. Locate a suitable immersion $\varphi$, then apply Proposition~\ref{prop:Ln} to create an infinite family $\{L_n\}$ of links that still satisfy the hypothesis of Generalized Property~R, but not necessarily the conclusion for most $n$. One can generalize still further, for example replacing the domain of $\varphi$ by a Klein bottle, although the authors presently have no explicit examples of this.

\begin{proof}[Proof of Proposition~\ref{prop:Ln}]
Since $\varphi$ is an immersion and $C_0$ is disjoint from all other $C_t$, there is a neighborhood of $C_0$ in $M$ intersecting $\im \varphi$ in an annulus. Let $C^*$ be a circle in this neighborhood, parallel to $C_0$ and disjoint from  $\im \varphi$. Thus, it is pushed off of $C_0$ by the $\pm 1$ framing in $M$. (All signs in the proof agree with the one in the statement of the theorem.) Let $L^*_n\subset M$ be the framed link obtained from $L_n$ by adding a $\pm 2$-framed component along $C_0$ and a dotted circle along $C^*$. We will show that $L_n^*$ is obtained from $L_n$ by adding a canceling Hopf pair and sliding handles, and that the links $L_n^*$ are all handle-slide equivalent. (As usual, the dotted circle is not allowed to slide over the other link components.) This will complete the proof.

To reduce $L^*_n$ to $L_n$, slide $C_0$ over the dotted circle $C^*$ using the simplest band connecting the parallel circles. This changes $C_0$ to a 0-framed meridian of $C^*$. (To compute the framing, recall that the dotted circle is automatically 0-framed, and that if $C_0$ and $C^*$ have parallel orientations, we are subtracting handles with linking number $\pm 1$.) The dotted circle $C^*$ bounds a disk $D\subset M$ and now has a 0-framed meridian, but this is not yet a canceling Hopf pair since the link $L$ may still intersect $D$. However, we can easily remove each intersection by sliding the offending strand of $L$ over the meridian. This does not change $L$ but liberates the Hopf pair as required.

To see that the links $L_n^*$ are handle-slide equivalent, we wish to isotope their common component $C$ at $C_0$ around the continuous family $C_t$, returning to the original position at $C_1=C_0$. By construction, $C$ will never meet the dotted circle at $C^*$, but it will meet components of $L_n$. Since the isotopy lies in $M'$, each encounter with $L'$ consists of a push across a surgery solid torus in $M'$, which is precisely a handle slide in $M$ of $C$ over a component of $L'$. On the other hand, $L''_n$ transversely intersects the path of $C$, once at each annulus $A_i$, so we must somehow reverse the crossing of $C$ with $L_n''$ there. Since the framing of $C$ has $\pm 1$ twist relative to $\varphi$, and hence to $A_i$, we can reverse the crossing by sliding $L_n''$ over $C$. This is shown by Figure~\ref{fig:cross} (or its mirror image if the sign is $-1$), where  $A_i$ is drawn as in Figure~\ref{fig:Ai}, but after a diffeomorphism so that $L_n''$ appears as a vertical line. (To see that the framing transforms as drawn, note that the right twist in the framing of the curve parallel to $C$ cancels the left twist introduced by the self-crossing added to $L_n''$ by the band-sum.) When $C$ returns to its original position at $C_1=C_0$, the link has been transformed to $L_{n\pm1}^*$ by handle slides, completing the proof.
\end{proof}

 \begin{figure}[ht]
 \labellist
\small\hair 2pt
\pinlabel  $A_i$ at 20 370
\pinlabel  $A_i$ at 20 200
\pinlabel  $A_i$ at 20 33
\pinlabel  $C$ at 65 390
\pinlabel  $C$ at 65 210
\pinlabel  $C$ at 65 65
\pinlabel  {${\rm pushed\; off} \; C$} at 55 450
\pinlabel  $L''_n$ at 105 480
\pinlabel  $L''_{n+1}$ at 100 147
\endlabellist
     \centering
    \includegraphics[scale=0.7]{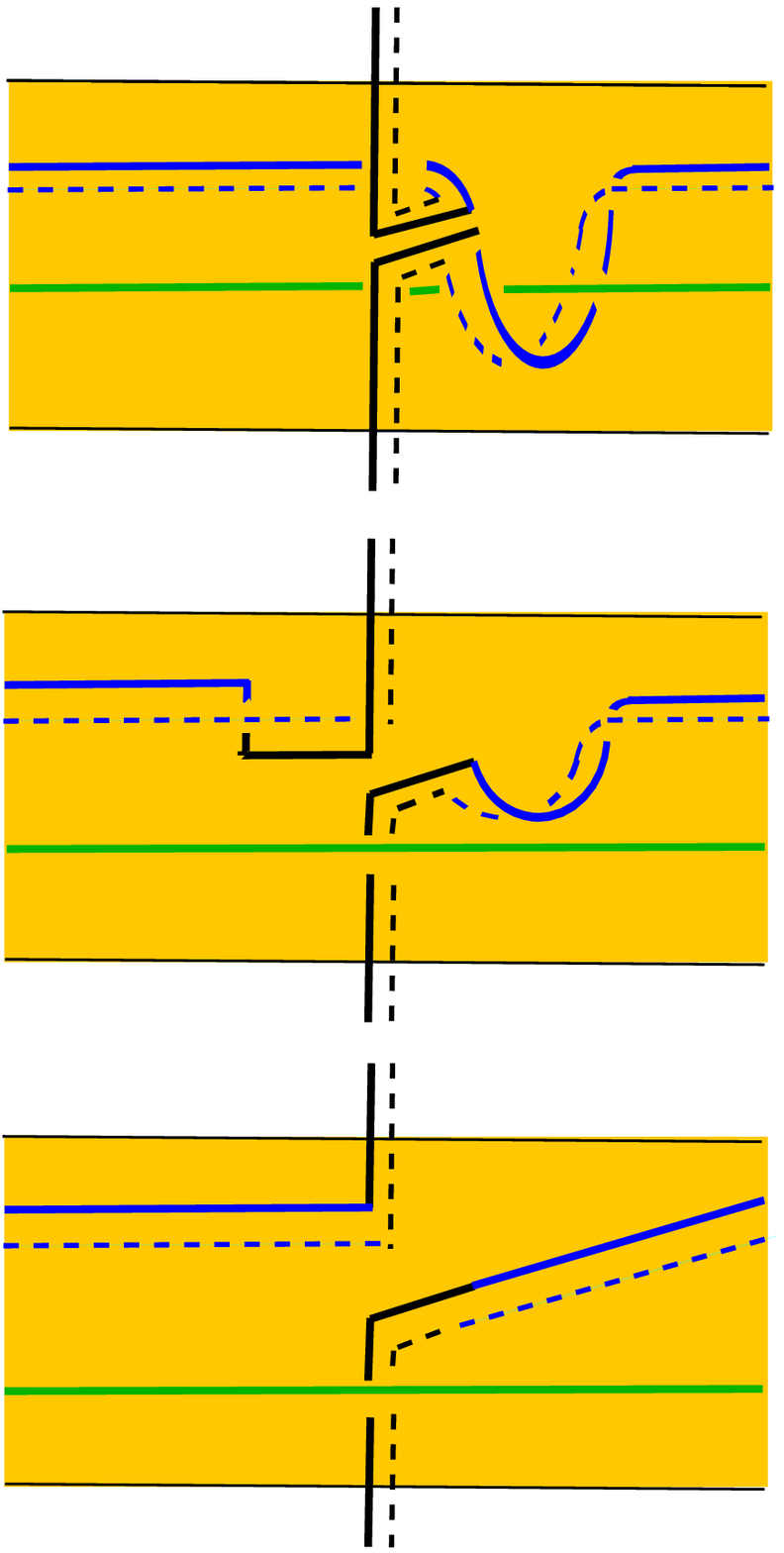}
    \caption{} \label{fig:cross}
    \end{figure}

The examples $L_{n,k}$ from Section~\ref{sect:nonstandard} form a family as in Proposition~\ref{prop:Ln} for each fixed $k$, generated by the handle-slide trivial case $n=0$. Figure \ref{fig:gompf6} shows the starting point.  It is basically the right side of Figure \ref{fig:Gompffig3b} with both a twist-box and one of the 0-framed links moved to a more symmetric position.

 \begin{figure}[ht]
  \labellist
\small\hair 2pt
\pinlabel  $k$ at 80 130
\pinlabel  $-k$ at 275 130
\pinlabel  $n$ at 170 160
\pinlabel  $-n$ at 170 50
\pinlabel  $[1]$ at 20 245
\pinlabel  $[-1]$ at 210 180
\pinlabel  $0$ at 175 115
\pinlabel  $0$ at 25 130
\endlabellist
     \centering
    \includegraphics[scale=0.7]{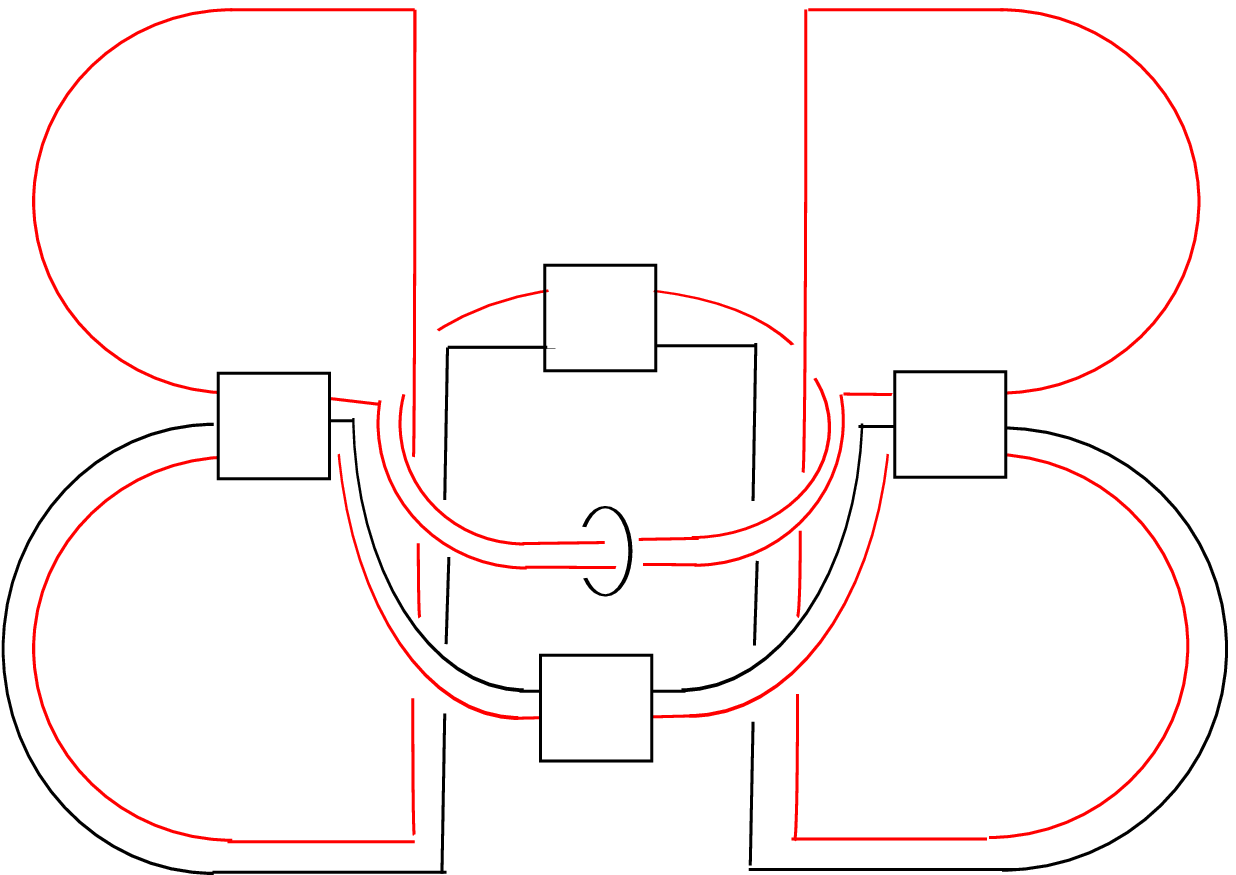}
    \caption{} \label{fig:gompf6}
    \end{figure}

Let $M$ be the manifold diffeomorphic to $S^3$ obtained by surgery on the bracketed (red) circles. (Recall that these form an unlink in $S^3$.) Let $L'$ and $L''_n$ be the small and large 0-framed circles, respectively. There is a vertical plane, perpendicular to the page, that bisects the picture.
This plane bisects both $\pm n$ twist boxes and contains $L'$. We interpret this plane as a 2-sphere in the background $S^3$. Surgery on $L'$ splits the sphere into a disjoint pair of spheres. Neither of these lies in $M'$ since each is punctured by both bracketed circles. However, we can remove the punctures by tubing the spheres together along both bracketed curves, obtaining an embedded torus in $M'$. The fibers $C_t$ of this torus include a meridian of each bracketed circle, and the torus induces the 0-framing on these relative to the background $S^3$. To see these in $M$, ignore the link $L_n$ and unwind the bracketed link. Blowing down reveals the 0-framed meridian of the $[\mp 1]$-framed curve to be a $\pm 1$-framed unknot in $M$ (see \cite[Figure 5.18]{GS}). Thus, we can choose either meridian to be $C_0$, with framing $\pm 2$ in $M$, i.e., $\pm 1$ in $S^3$. (In either case, the framing in $S^3$ is preserved as we travel around the torus, so it becomes 0 in $M$ when we reach the other meridian, showing that we cannot get off the ride in the middle.) Since $L_n''$ intersects the torus at two points, at the $\pm n$-twist boxes, it is easy to verify that the procedure turns $L_n^*$ into $L_{n\pm 1}^*$, with the sign depending both on the choice of $C_0$ and the direction of motion around the torus. It is instructive to view the procedure explicitly as a sequence of four handle slides in Figure~\ref{fig:gompf6}.

It is not completely obvious how to trivialize the link by handle slides when $n=0$. In the $k=1$ case, Figures \ref{fig:Gompffig4} and \ref{fig:Gompffig5} isotope the link to Figure \ref{fig:squareknot}, solving the problem with a single handle slide. For general $k$, it is helpful to begin from a somewhat simpler picture of $L_{n,k}$, Figure~\ref{fig:example}, which is obtained from Figure \ref{fig:gompf6} by an isotopy that rotates both $\pm k$ twist boxes 180 degrees outward about vertical axes. Now set $n = 0$ and consider the two handle slides in Figure~\ref{fig:example2} (following the arrows, framed by the plane of the paper).  The slides first complicate $L'$, but after an isotopy shrinking the strands shown in green, $L'$ simplifies,  so that a slide over the $[-1]$-framed circle changes it to a meridian of the $[+1]$-framed circle. The $\pm k$-twist boxes then cancel, showing that the link is trivial.

\begin{figure}[ht]
     \labellist
\small\hair 2pt
\pinlabel  $k$ at 75 150
\pinlabel  $-k$ at 275 150
\pinlabel  $n$ at 170 100
\pinlabel  $-n$ at 170 20
\pinlabel  $[1]$ at 25 220
\pinlabel  $[-1]$ at 220 220
\pinlabel  $0$ at 180 200
\pinlabel  $0$ at 25 100
\endlabellist
 \centering
    \includegraphics[scale=0.7]{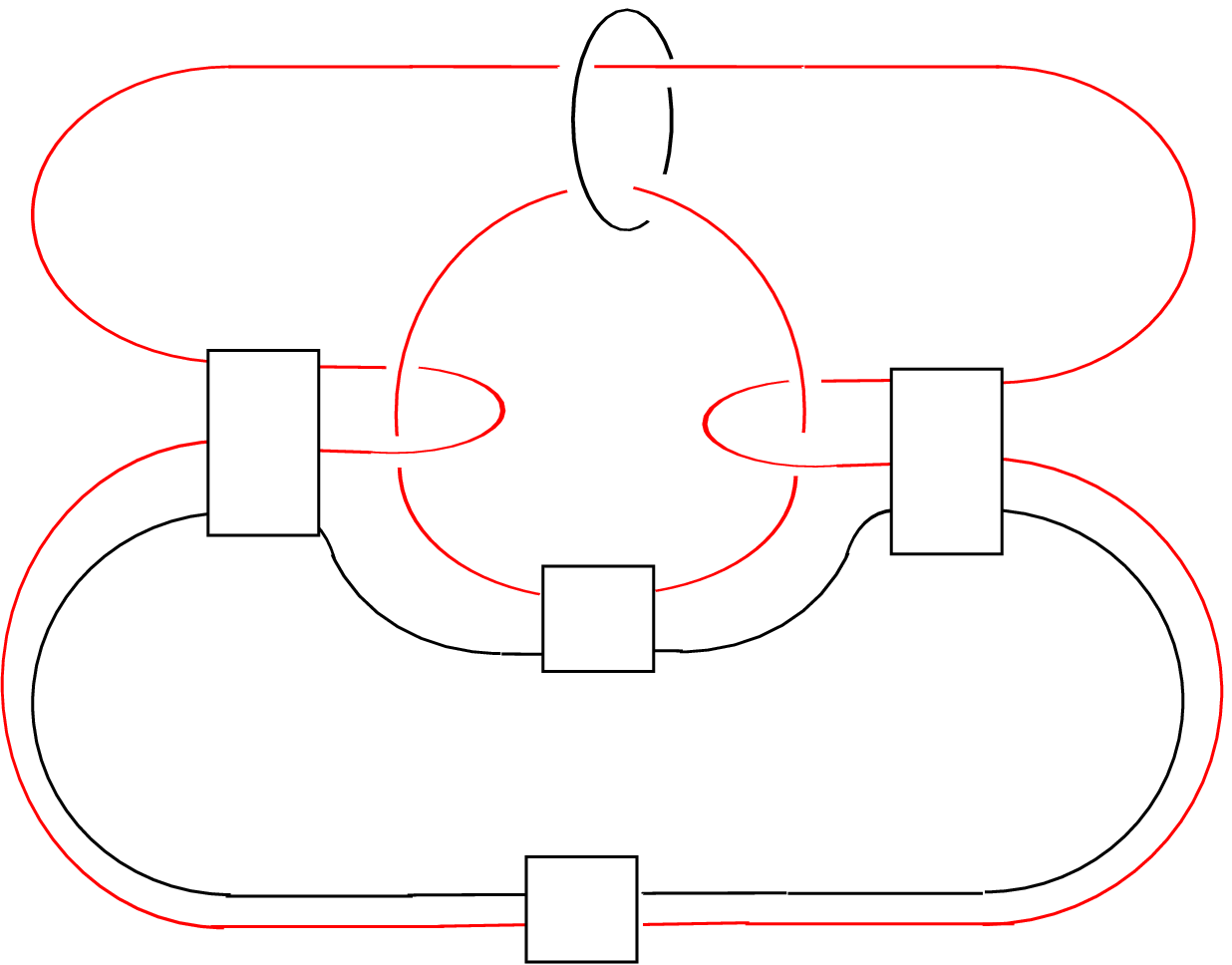}
    \caption{} \label{fig:example}
    \end{figure}

    \begin{figure}[ht]
     \labellist
\small\hair 2pt
\pinlabel  $0$ at 105 433
\pinlabel  $0$ at 105 387
\endlabellist
     \centering
    \includegraphics[scale=0.7]{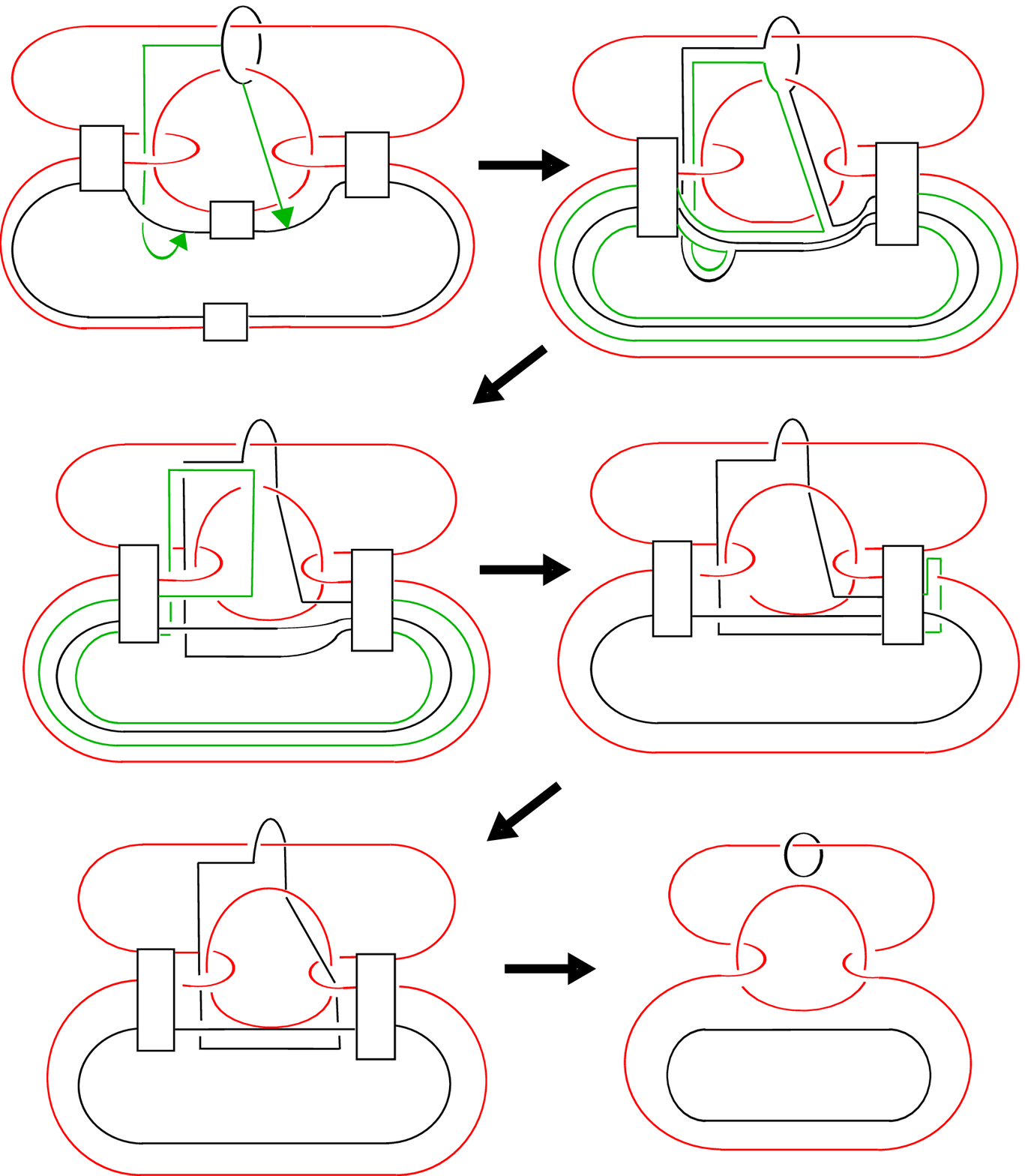}
    \caption{} \label{fig:example2}
    \end{figure}

This construction was derived from the dual construction in \cite{Go1}, showing that the relevant homotopy balls are standard by introducing a canceling 2-handle/3-handle pair. While it was not easy to explicitly dualize the construction, in the end the two constructions appear rather similar. The construction in \cite{Go1} also involves introducing a new link component and moving it around a torus. In that case, the new component was a 0-framed unknot, the attaching circle of the new 2-handle.  Akbulut's recent proof \cite{Ak} that infinitely many Cappell-Shaneson homotopy 4-spheres are standard again relies on such a trick. In each case, the key seems to be an embedded torus with trivial
normal bundle, to which a 2-handle is attached with framing $\pm 1$, forming
a {\em fishtail neighborhood}. This neighborhood has self-diffeomorphisms
that can undo certain cut-and-paste operations on the 4-manifold, as we implicitly
saw in the case of Proposition~\ref{prop:Ln} with $\varphi$ an embedding.
While this procedure has been known in a different context for at least
three decades, underlying the proof that simply connected elliptic surfaces with fixed $b_2$ are
determined by their multiplicities, it still appears to be underused. In
fact, one can obtain a simpler and more general proof that many Cappell-Shaneson
4-spheres are standard by going back to their original definition and
directly locating fishtail neighborhoods \cite{Go2}.

\end{document}